\documentclass[a4paper, 10pt,oneside]{amsart}
\usepackage{amsmath, amssymb, amsfonts}
\usepackage[all]{xy}
\newtheorem{theorem}{Theorem}[section]
\newtheorem{lemma}[theorem]{Lemma}
\newtheorem{proposition}[theorem]{Proposition}

\theoremstyle{definition}

\newtheorem{example}[theorem]{Example}

\theoremstyle{remark}
\newtheorem{remark}[theorem]{Remark}
\newcommand{\g}{\mathfrak{g}}
\newcommand{\gtl}{\tilde{\mathfrak{g}}}

\newcommand{\h}{\mathfrak{h}}

\newcommand{\n}{\mathfrak{n}}
\newcommand{\ntl}{\tilde{\mathfrak{n}}}

\newcommand{\Z}{\mathbb{Z}}
\newcommand{\C}{\mathbb{C}}
\newcommand{\N}{\mathbb{N}}

\newcommand{\lam}{\Lambda_{\ell-1}}
\newcommand{\llam}{\Lambda_{\ell}}

\newcommand{\vac}    {\ensuremath{{\bf 1}}}

\numberwithin{equation}{section}

\begin{document}

\title{Combinatorial bases of Feigin-Stoyanovsky's type subspaces of level 2 standard modules for $D_4^{(1)}$}

\author{Ivana Baranovi\'c}
\address{Faculty of Chemical Engineering and Technology, University of Zagreb, Maruli\'cev trg 19, Zagreb, Croatia}
\curraddr{}
\email{ibaranov@fkit.hr}
\thanks{}

\subjclass[2000]{Primary 17B67; Secondary 17B69, 05A19.\\ \indent Partially supported by the Ministry of Science and Technology of the Republic of Croatia, Project ID 037-0372794-2806}
\date{}
\dedicatory{}

\begin{abstract}
Let $\gtl$ be an affine Lie algebra of type $D_{\ell}^{(1)}$ and $L(\Lambda)$ its standard module with a highest weight vector $v_{\Lambda}$. For a given $\Z$-gradation $\gtl = \gtl_{-1} + \gtl_0 + \gtl_1$, we define Feigin-Stoyanovsky's type subspace as
$$W(\Lambda) = U(\gtl_1) \cdot v_{\Lambda}.$$ 
By using vertex operator relations for standard modules we reduce the Ponicar\'{e}-Brikhoff-Witt spanning set of $W(\Lambda)$ to a basis and prove its linear independence by using Dong-Lepowsky intertwining operators. 
\end{abstract}

\maketitle

\section{Introduction}
J. Lepowsky and R.L. Wilson initiated a program whose goal was to study representations of affine Lie algebras by means of vertex operators and to find Rogers-Ramanujan type combinatorial bases of these representations (\cite{LW1}, \cite{LW2}). Using this approach, Lepowsky and Primc obtained in \cite{LP} new character formulas for standard $\mathfrak{sl}(2, \C)^{\widetilde{}}$-modules. In \cite{FS} Feigin and Stoyanovsky gave another proof and combinatorial interpretation of these formulas. One of the new ingredients introduced in the paper were so-called principal subspaces. These subspaces are generated by the affinization of the nilpotent subalgebra $\n_+$ of $\g$ from the triangular decomposition $\g = \n_{-} \oplus \h \oplus \n_+$. Using the intertwining operators, G. Georgiev (\cite{G}) proved the linear independence of quasi-particle bases of Feigin-Stoyanovsky level $k$ principal subspaces for $\mathfrak{sl}(\ell + 1, \C)$. 

Similar approach was taken by S. Caparelli, Lepowsky and A. Milas in \cite{CLM1} and \cite{CLM2}. They obtained Rogers-Selberg recursions for characters of level $k$ principal subspaces for $\mathfrak{sl}(2, \C)$ using the so-called Dong-Lepowsky intertwining operators (\cite{DL}). As a continuation of this work, C. Calinescu obtained reccurence system for the characters of principal subspaces of level 1 standard modules for $\mathfrak{sl}(\ell + 1, \C)^{\widetilde{}}$ (\cite{C1}) and of certain higher-level standard modules for $\mathfrak{sl}(3, \C)^{\widetilde{}}$ (\cite{C2}). 

In \cite{P1} M. Primc studied an object similar to the principal supbspace, the so-called Feigin-Stoyanovsky's type subspace $W(\Lambda)$ for $\g = \mathfrak{sl}(\ell + 1, \C)$. He constructed a combinatorial base for this space using Schur functions. A basis of the whole $L(\Lambda)$ could then be obtained as an inductive limit of translations of the basis of $W(\Lambda)$ by a certain Weyl group element. This was done for a particular choice of gradation of $L(\Lambda)$ and for any dominant $\Lambda$. It remained unclear how to extend this approach to other algebras. In \cite{P2}, Primc found combinatorial basis of $W(\Lambda)$ for all classical Lie algebras and any possible gradation but only for basic modules $L(\Lambda_0)$. In this paper, he applied crystal base character formula (\cite{KKMMNN}) for level one modules but level $k > 1$ modules were not discussed.

Using the Caparelli-Lepowsky's-Milas' approach via intertwining operators and the description of the basis from \cite{FJLMM}, Primc gave in \cite{P3} a simpler proof of linear independence for the basis of $W(\Lambda)$ constructed in \cite{P1}. In the similar manner, G. Trup\v cevi\'c obtained in \cite{T1} and \cite{T2} combinatorial bases of Feigin-Stoyanovsky's type subspaces $W(\Lambda)$ for all standard  modules of any level or gradation for $\mathfrak{sl}(\ell + 1, \C)$ whilst M. Jerkovi\'c managed to obtain the systems of recurrence relations for formal characters of Feigin-Stoyanovsky's type subspaces of level $k$ standard modules for $\mathfrak{sl}(\ell + 1, \C)^{\widetilde{}}$ (\cite{J}).

In this paper we extend the same approach to construct a combinatorial basis for Feigin-Stoyanovsky's type subspace of four standard level 1 modules for algebra of type  $D_{\ell}^{(1)}$. Then we construct a combinatorial basis for Feigin-Stoyanovsky's type subspace of the level 2 standard modules for $D_{4}^{(1)}$ type algebra. 

Let $\g$ be a simple Lie algebra, $\h$ a Cartan subalgebra of $\g$ and $R$ the corresponding root system. We have a root space decomposition $\g = \h + \sum_{\alpha \in R} \g_{\alpha}$ and we fix root vectors $x_{\alpha}$. For a minuscule co-weight $\omega$ we denote by $\Gamma$ the set of all roots such that $\alpha(\omega) = 1$. Then

\begin{align}
 \g = \g_{-1}  \oplus \g_0 \oplus \g_1 
\end{align}

\noindent is a $\Z$-gradation of $\g$ for $\g_{\pm 1} = \sum_{\alpha \in \Gamma} \g_{\pm \alpha}$, $\g_0 = \h + \sum_{\omega(\alpha)=0} \g_{\alpha}$

To $\g$ we associate the affine Lie algebra, $\gtl = \g \otimes \C c \otimes \C d $, with the canonical central element $c$, degree operator $d$ and fixed real root vectors $x_{\alpha}(n) = x_{\alpha} \otimes t^n$. Gradation of $\g$ induces $\Z$-gradation of $\gtl$:

$$ \gtl = \gtl_{-1}  \oplus \gtl_0 \oplus \gtl_1.$$

 $\gtl_1 = \g_1 \otimes \C [t, t^{-1}]$ has a basis

\begin{align}
 \{x_{\gamma}(j) \ \vert \ j \in \Z, \gamma \in \Gamma \ \}.
\end{align}

Let $L(\Lambda)$ be a standard $\gtl$-module of level $k=\Lambda(c)$ and let $v_{\Lambda}$ be its highest weight vector. A Feigin-Stoyanovsky's type subspace is a $\gtl_1$-submodule of $L(\Lambda)$ generated with $v_{\Lambda}$,
\begin{align}
W(\Lambda) = U(\gtl_1) \cdot v_{\Lambda}.
\end{align}
The goal is to find a combinatorial basis of $W(\Lambda)$ whose elements are of the form $x(\pi) v_{\Lambda}$ where $x(\pi)$ are \textit{monomials }consisting of $\{ \ x_{\gamma}(-j) \ \vert j \in \N, \gamma \in \Gamma \}$ factors. We call such basis also \textit{a monomial basis}. 

This paper is organized as follows: in Sections \ref{algebra} - \ref{spinor} we give the setting. In Section \ref{feigin} we define the Feigin-Stoyanovsky subspace and monomial ordering on the set of monomials.  

In Sections \ref{uvjeti_razlike} - \ref{kraj1} we first study $W(\Lambda)$ when $\Lambda$ is of level 1. By Poincar\'e-Birkhoff-Witt's theorem the set of \textit{monomial vectors} 
\begin{align}\label{spanning}
\{ \ x_{\beta_{1}}(-j_1) \dots x_{\beta_{r}}(-j_r)v_{\Lambda} \ \vert \ \beta_{i} \in \Gamma, r \in \Z_+, j_r \in \N \ \}
\end{align}
spans $W(\Lambda)$. We reduce set \eqref{spanning} using the relations between vertex operators. We express the smallest element with regard to the linear ordering we established in \ref{order} as the sum of the greater ones. Therefore this smallest element can be excluded from the spanning set \eqref{spanning}. The elements that do not contain this smallest monomials are said to satisfy so-called difference and initial conditions on $W(\Lambda)$. Such elements span $W(\Lambda)$. In Section \ref{kraj1} we show their linear independence and it follows these monomials form a basis of the Feigin-Stoyanovsky's type subspace. In the proof of the linear independence, we use the operators which commute with the action of $\gtl_1$ and send certain vectors from one $W(\Lambda^1)$ subspace to another $W(\Lambda^2)$ subspace. This operators will be the specific coefficients of the intertwining operators (see \cite{DL}). We also use the operator $e(\lambda)$ which has the property to lift the degree of the monomial (cf. \cite{DLM}), i.e.
\begin{align} \label{operatoric}
x_{\alpha}(n)e(\lambda)=e(\lambda)x_{\alpha}(n + \langle \lambda, \alpha\rangle), \qquad n \in \Z.
\end{align}
Using the intertwining operators and the operator \eqref{operatoric} we prove the linear independence by induction on the degree of the monomial on all subspaces of level 1 modules simultaneously. 

In Sections \ref{pocetak2} - \ref{kraj2} we study Feigin-Stoyanovsky's type subspaces of $L(\Lambda)$ where $L(\Lambda)$ is a fundamental level two module for algebra of type $D_4^{(1)}$. We also describe the monomials that satisfy the difference and initial conditions on such $W(\Lambda)$. They span $W(\Lambda)$ and we prove their linear independence using the intertwining operators and the simple current operator. 

As for generalization of the result, i.e. the construction of the combinatorial basis for $W(\Lambda)$ when $\Lambda$ is a level two fundamental weight in cases $\ell > 4$, the main problem remains the proof of the linear indpenedence of the spanning set. Namely, it is at the moment unclear how to construct the adequate intertwining operators. 

I thank Mirko Primc for his help on the topic and valuable suggestions. 

\section{Affine Lie algebras and basic modules}\label{algebra}

For $\ell \geq 4$ let $\g$ be a simple Lie algebra of type $D_{\ell}$, $\h$ a Cartan subalgebra of $\g$ and $R$ the corresponding root system. Fix a basis $\Pi = \{\alpha_1, \dots, \alpha_l \}$ of $R$ where 
$\alpha_1 = \epsilon_1 - \epsilon_2, \dots , \alpha_{\ell-1} = \epsilon_{\ell-1} - \epsilon_{\ell}$,  $\alpha_{\ell} = \epsilon_{\ell-1} + \epsilon_{\ell}$ and $\epsilon_i, \ i=1, \dots, \ell$ are the usual orthonormal unit vectors. Then we have the corresponding triangular decomposition, $\g = \n_{-} \oplus \h \oplus \n_+$.

We denote by $\langle \cdot,\cdot\rangle$ the Killing form on $\g$ which enables us to identify $\h \cong \h^*$. Let $\theta$ be the maximal root, i.e. $\theta = \epsilon_1 + \epsilon_2$. Then the Killing form can be  normalized so that $\langle \theta,\theta\rangle = 2$. 

For each root $\alpha$ we fix a root vector $x_{\alpha}$. Let $Q=Q(R)$ and $P=P(R)$ be the root and weight lattice of $\g$ respectively.

Denote by $\omega_1, \dots, \omega_{\ell}$ the fundamental weights and define, for convenience,  $\omega_0 = 0$. In this paper we fix $\omega = \omega_1=\epsilon_1$. Notice that $\omega$ is a minuscule weight, i.e. 
$$\{\langle \omega, \alpha \rangle \ \vert \ \alpha \in R \} = \{-1, 0 ,1 \}.$$
This gives us the induced $\Z$-gradation of $\g$:
\begin{align}\label{grad}
\g = \g_{-1} + \g_0 + \g_1.
\end{align}
Set $\Gamma = \{ \alpha \in R \ \vert \ \langle \omega, \alpha  \rangle = 1\}.$
Then $$\Gamma = \{ \gamma_{2}, \dots, \gamma_{\ell}, \gamma_{\underline{\ell}}, \dots, \gamma_{\underline{2}} \}$$
where $\gamma_i = \epsilon_1 + \epsilon_i$  and $\gamma_{\underline{i}} = \epsilon_1 - \epsilon_i$, $i=1, \dots, \ell$.  Notice that $\gamma_2$ is the maximal root, $\gamma_2 = \theta$.

Denote by $\gtl$ the associated affine Lie algebra (see, for example, \cite{K}),
$$\gtl = \g \otimes \C [t, t^{-1}] \oplus \C c \oplus \C d$$
where $c$ is the canonical central element and $d$ the degree operator. Set $x(j)=x \otimes t^{j}$ for $x \in \g$, $j \in \Z$ and let $x(z)=\sum_{n \in \Z} x(n)z^ {-n-1}$ be the formal Laurent series in formal variable $z$.
Set $\h^e = \h \oplus \C c \oplus \C d$, $\ntl_{\pm} = \g \otimes t^{\pm 1} \C [t^{\pm 1}] \oplus \n_{\pm}$. We have a triangular decomposition: $\gtl = \ntl_{-} \oplus \h^e \oplus \ntl_{+}$ and the induced $\Z$-gradation with respect to $\omega$:

$$\gtl=\gtl_{-1} + \gtl_0 + \gtl_1.$$

Let $\widetilde{\Pi} = \{\alpha_0, \alpha_1, \cdots, \alpha_{\ell} \} \subset (\h^e)^*$ be the set of simple roots of $\gtl$. We define the fundamental weights $\Lambda_i \subset (\h^e)^*$ by 
$\langle \Lambda_i, \alpha_j \rangle = \delta_{ij}$ and $\Lambda_i (d)=0$ where $\langle \cdot, \cdot \rangle$ is the standard extension of Killing form onto $\h^e$. 

Let $L(\Lambda)$ be a standard $\gtl$-module, i.e. an irreducible highest weight module with dominant integral highest weight 

\begin{align}\label{tezina}
\Lambda = k_0 \Lambda_0 + k_1 \Lambda_1 + \dots + k_l \Lambda_{\ell}, \quad k_i \in \Z_{+}, \ i=1, \dots, \ell.
\end{align}
The central element $c$ acts on $L(\Lambda)$ as a scalar $k$
which is called the level of module $L(\Lambda)$. Modules $L(\Lambda_0), L(\Lambda_1), L(\Lambda_{\ell-1})$ and $L(\Lambda_{\ell})$ are of level $1$, whilst $L(\Lambda_i), i=2, \dots, \ell -2$ are level $2$ modules. In general, module $L(\Lambda)$ with $\Lambda$ as in \eqref{tezina} is of the level
$$k=k_0 + k_1 + 2k_2 + \dots + 2k_{\ell-2} + k_{\ell-1} + k_{\ell}.$$

\section{Feigin-Stoyanovsky's type subspaces $W(\Lambda)$}\label{feigin}
Let $L(\Lambda)$ be a standard $\gtl$-module with a dominant integral highest weight
$$\Lambda = k_0 \Lambda_0 + k_1 \Lambda_1 + \dots + k_{\ell} \Lambda_{\ell}.$$
Then $L(\Lambda)$ is of level $k=k_0 + k_1 + 2 k_2 + \dots +2 k_{\ell-2} + k_{\ell-1} + k_{\ell}$. Fix a highest weight vector $v_{\Lambda}$.

Recall the induced $\Z$-gradation we have on $\gtl$ with respect to fundamental weight $\omega=\omega_1$:
$$\gtl = \gtl_{-1} + \gtl_0 + \gtl_1.$$
We define Feigin-Stoyanovsky's type subspace as
$$W(\Lambda) = U(\gtl_1) \cdot v_{\Lambda} \subset L(\Lambda).$$
If we further on set 
$$\gtl_1^{-}=\gtl_1 \cap \ntl_{-} = \g_1 \otimes t^{-1} \C [t^{-1}]$$
then obviously
$$W(\Lambda) = U(\gtl_1^{-}) \cdot v_{\Lambda}.$$
By Poincar\'e-Birkhoff-Witt theorem, the spanning set of $W(\Lambda)$ consists of monomial vectors
\begin{align}\label{PBW}
\{ x_{\delta_1}(-n_1) x_{\delta_2}(-n_2) \dots x_{\delta_r}(-n_r)v_{\Lambda} \ \vert \ r \in \Z_+, \ n_i \in \N, \ \delta \in \Gamma  \ \}.
\end{align}
Elements of this spanning set can be identified with the monomials from $U(\gtl_1) \cong S(\gtl_1)$. Therefore we often refer to the elements of the form $ x_{\delta}(-j)$,  $\delta_i \in \Gamma, \ n \in \N,$ as \textit{the variables, elements or factors of the monomial}.

\section{Monomial ordering}\label{order}
To reduce the spanning set \eqref{PBW} of $W(\Lambda)$ to a basis, we establish a linear order on the set of monomials. The monomials from $U(\gtl_1)$ can be interpreted as so-called \textit{colored partitions}: $\pi : \{ x_{\delta}(-j) \ \vert \ \delta \in \Gamma, \ j \in \Z \} \longrightarrow \Z_+$ denotes the monomial
$$x(\pi) = x_{\delta_{1}} (-j_1)^{\pi(x_{\delta_{1}}(-j_1))} x_{\delta_{2}}(-j_2)^{\pi(x_{\delta_{2}}(-j_2))} \dots x_{\delta_r}(-j_r)^{\pi(x_{\delta_{r}}(-j_r))}.$$
$\h$-weights $\delta_{i} \in \Gamma$ are referred to as \textit{colors}. For color $\gamma = \epsilon_1 \pm \epsilon_j$ we will define \textit{the oposite color} as $\underline{\gamma} = \epsilon_1 \mp \epsilon_j$. We will also sometimes call $\gamma_{\underline{i}}$, $i=2, \dots, \ell$, \textit{negative} and $\gamma_i$, $i=2, \dots, \ell$, \textit{positive colors}.
In further text we use the label $x(\pi)$ for the monomials from $U(\gtl_1)$.

Next we define \textit{the shape} of the partition $\pi$. Partition $\pi : \{ x_{\delta}(-n) \ \vert \ \delta \in \Gamma, n \in \N \}$ (i.e. monomial $x(\pi)$) is of the shape
\begin{align*}
& s(\pi)  :  \Z \longrightarrow \Z_{+} \\
& s(\pi)(j)  =  \sum_{\delta \in \Gamma} \pi(x_{\delta}(-j))
\end{align*}
which can be graphically displayed. 
\begin{example}
Shape of the monomial $x(\pi)=x_{\delta_{4}}(-4)x_{\delta_{3}}(-4)x_{\delta_{2}}(-3)x_{\delta_{1}}(-2)$ looks like:
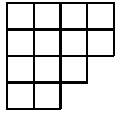
\begin{center}\begin{picture}(100,50)(-30,-25) \thinlines
\multiput(0,0)(0,10){3}{\line(1,0){40}}
\multiput(0,-10)(0,10){1}{\line(1,0){30}}
\multiput(0,-20)(0,10){1}{\line(1,0){20}}
\multiput(0,-20)(0,10){1}{\line(0,1){40}}
\multiput(10,-20)(0,10){1}{\line(0,1){40}}
\multiput(20,-20)(0,10){1}{\line(0,1){40}}
\multiput(30,-10)(0,10){1}{\line(0,1){30}}
\multiput(40,0)(0,10){1}{\line(0,1){20}}
\end{picture}\end{center}
\end{example}
We define that $s(\pi) < s(\pi')$ if there exist $j_0 \in \N$ such that $s(\pi)(j) = s(\pi')(j)$ for $j < j_0$ and $s(\pi)(j_0) < s(\pi')(j_0)$. 

We further define order on the set of colors, i.e. on $\Gamma$. We set:
$$\gamma_2 > \gamma_3 > \dots > \gamma_{\ell} > \gamma_{\underline{\ell}} > \dots > \gamma_{\underline{2}}.$$

Now, on the set of factors $x_{\gamma}(-j)$ we define

$$x_{\delta}(-i) < x_{\tau}(-j) \qquad \textrm{if} \qquad \left \{ \begin{array}{l}
-i<-j \textrm{\quad or}\\
i=j \textrm{\quad i \quad} \delta < \tau.
\end{array} \right.$$
Since elements from $U(\gtl_1) = S(\gtl_1)$ commute, we can assume from now on that elements in monomial $x(\pi)$ are sorted descending from right to left. Order on monomials is then defined as follows: let 
\begin{align*}
x(\pi) &= x_{\delta_r}(-i_r)x_{\delta_{{r-1}}} (-i_{r-1})\dots x_{\delta_2}(-i_2)x_{\delta_1}(-i_1) \quad \textrm{and} \\
x(\pi') &= x_{\delta_s'}(-i_s)x_{\delta_{s-1}'}(-i_{s-1}') \dots x_{\delta_2'}(-i_2')x_{\delta_1'}(-i_1').
\end{align*}
Then $x(\pi) < x(\pi')$ if 
\begin{itemize}
\item[-] $s(\pi) < s(\pi')$ (we first compare shapes) or 
\item[-] $s(\pi) = s(\pi')$ and there exist $j_0 \in \N$ such that $\delta_j = \delta_j'$ for all $j < j_0$ and $\delta_{j_0} < \delta_{j_0}'$ (if monomials are of the same shape, we compare their colors from right to left). 
\end{itemize}
Such order is compatible with multiplication in $S(\gtl_1)$:
\begin{proposition}\label{uredjaj}
Let $x(\delta_1) \leq x (\delta_2)$ and $x(\mu_1) \leq x(\mu_2)$. Then $x(\delta_1) x(\delta_2) \leq x(\mu_1)x(\mu_2)$ and if any of the first two inequalities is strict, then so is the last one.
\end{proposition}

\begin{proof}
We divide the proof in three main cases:
\begin{itemize}
 \item[1)] both inequalities $x(\delta_1) \leq x (\delta_2)$ and $x(\mu_1) \leq x(\mu_2)$ come from shape comparison,
\item[2)] one of the inequalities comes from shape and the other from color comparison,
\item[3)] both inequalities come from color comparison.
\end{itemize}
Cases 1) and 2) are straightforward, and the third case is an example of a merge sort, see for instance \cite{T1}.
\end{proof}

\section{Vertex operator construction of level one modules}

In the paper we use the well known Frenkel-Kac-Segal vertex operator algebra construction of the standard $\gtl$-modules of level 1, \cite{FK}, \cite{S}. The brief outline of the construction follows and the details can be found in \cite{FLM}, \cite{DL} or \cite{LL}. We use the notation from \cite{LL}.

Let $M(1)$ denote the Fock space for the homogenous Heisenberg subalgebra and let $\C [P]$ be the group algebra of the weight lattice $P$ with a basis $e^{\lambda}, \lambda \in P$. Set $V_Q = M(1) \otimes \C [Q]$ and $V_P = M(1) \otimes \C [P]$. The action of Heisenberg subalgebra on the tensor product $V_P$ extends to the action of Lie algebra $\gtl$ via the vertex operator formula 
\begin{align}\label{vertex}
x_{\alpha}(z)= Y(e^{\alpha},z) = E^{-}(- \alpha,z) E^{+}(-\alpha,z) e_{\alpha} z^{\alpha}
\end{align}
for properly chosen root vectors $x_{\alpha}$. Here $z^{\alpha} = 1 \otimes z^{\alpha}$, $z^{\alpha}e^{\lambda} = z^{\langle \alpha, \lambda \rangle}$ and
$$E^{\pm}(\alpha,z) = E^{\pm}(\alpha,z) \otimes 1 = \textrm{exp} \left( \sum_{n>0} \frac{\alpha(\pm n) z^{\mp n}}{\pm n}  \right) \otimes 1.$$

On $V_P$ we have the following gradation: a homogenous vector $v = h(-n_1) \dots h(-n_r) \otimes e^{\lambda}$ is of degree
\begin{align}\label{gradacija}
\textrm{deg} (v) = -n_1 - n_2 - \dots -n_r - \frac{1}{2} \langle \lambda, \lambda \rangle .
\end{align}

Then, as a $\gtl$-module
$$M(1) \otimes \C [P] = L(\Lambda_0) + L(\Lambda_1) + L(\Lambda_{\ell-1}) + L(\Lambda_{\ell}).$$
More precisely,
$$L(\Lambda_0) \cong V_Q, \quad L(\Lambda_1) \cong e^{\omega_1} V_Q = e^{\omega}V_Q, \quad L(\lam) \cong V_Q e^{\omega_{\ell-1}}, \quad L(\llam) \cong V_Q e^{\omega_{\ell}} $$
with highest weight vectors $v_{\Lambda_0} = e^{\omega_0} = \vac, \ v_{\Lambda_1} = e^{\omega_1}, \ v_{\lam} = e^{\omega_{\ell-1}}$ and $v_{\llam} = e^{\omega_{\ell}}$ respectively. The space $V_Q$ has the structure of the vertex operator algebra (VOA) and $V_P$ is a module for this algebra (cf. \cite{FLM}, \cite{DL}).

Vertex operators $Y(e^{\lambda},z), \ \lambda \in P$ satisfy generalized Jacobi identity. We shall use the intertwining operators:
$$\mathcal{Y}: V_P \longrightarrow (\textrm{End} V_P) \{ z \},$$
$$v \mapsto \mathcal{Y}(v,z),$$
defined for $v=v^{*}\otimes e^{\mu} \in V_P$ as:
\begin{align} \label{vert_def}
\mathcal{Y}(v,z) = Y(v,z)e^{i \pi \mu}c(\cdot, \mu).
\end{align}
Operators $Y(e^{\lambda},z_1)$ and $\mathcal{Y}(e^{\mu}, z_2)$, $\mu, \lambda \in P$, satisfy the ordinary Jacobi identity and restrictions of $\mathcal{Y}(e^{\mu},z)$ are in fact maps

\begin{align}\label{intop}
\mathcal{Y}(e^{\mu},z) : L(\Lambda_i) \longrightarrow L(\Lambda_j)
\end{align}
if $\mu + \omega_i \equiv \omega_j \ \textrm{mod} \ Q.$ This restrictions give us the intertwining operators between standard modules of level 1 (\cite{DL}).

The commutator formula (\cite{DL})
\begin{align}\label{komutator}
[ Y(e^{\gamma},z_1), \mathcal{Y}(e^{\mu}, z_2) ] = \textrm{Res}  
\end{align}
and the relation
\begin{align}\label{formula}
\mathcal{Y}(e^{\lambda}, z_0) e^{\nu} = Ce^{\lambda + \nu}z_0^{\langle \lambda, \nu \rangle} + \dots \ \in z_0^{\langle \lambda, \nu \rangle} V_P [[z_0]]
\end{align}
show that
\begin{align}\label{komutiranje}
[Y(e^{\gamma},z_1) , \mathcal{Y}(e^{\mu},z_2)] = 0
\end{align}
if and only if
\begin{align}\label{uvjet_kom}
 \langle \gamma, \mu \rangle \geq 0 \qquad \textrm{for all} \ \gamma \in \Gamma.
\end{align}

\section{Simple current operator}\label{op_pr_struje}
For $\lambda \in P$ we denoted by $e^{\lambda}$ the multiplication operator $1 \otimes e^{\lambda}$ on $V_P= M(1) \otimes \C [P]$. Let 

\begin{align*}
e(\lambda): V_P \longrightarrow V_P \\
e(\lambda) = e^{\lambda} \epsilon(\cdot, \lambda).
\end{align*}

Then $e(\lambda)$ is obviously a linear bijection and using the vertex operator formula \eqref{vertex} we have the commutation relation:
$$Y(e^{\alpha},z)e(\lambda)=e(\lambda)z^{\langle \lambda, \alpha \rangle}Y(e^{\alpha},z), \qquad \textrm{for} \ \alpha \in R$$
which, in terms of components, gives:
\begin{align}\label{omega}
x_{\alpha}(n)e(\lambda)=e(\lambda)x_{\alpha}(n + \langle \lambda, \alpha\rangle), \qquad n \in \Z.
\end{align}
If we put $\lambda = \omega$ and $\gamma \in \Gamma$ in \eqref{omega}, we get
$$x_{\gamma}(n)e(\omega)=e(\omega)x_{\gamma}(n+1).$$
i.e. commuting of the $e(\omega)$ raises degree of factor $x_{\gamma}(n)$ by one. We will call $e(\omega)$ a simple current operator (\cite{DLM}). The restricitions of this operator give bijections between the level 1 modules:
\begin{align}
e(\omega) &: L(\Lambda_0) \longrightarrow L(\Lambda_1), \\
e(\omega) &: L(\Lambda_1) \longrightarrow L(\Lambda_0), \\
e(\omega) &: L(\Lambda_{\ell-1}) \longrightarrow L(\Lambda_{\ell}) \quad \textrm{and} \\
e(\omega) &: L(\Lambda_{\ell}) \longrightarrow L(\Lambda_{\ell-1}).
\end{align}
Namely, we have the following relations:
\begin{align}
e(\omega) & v_{\Lambda_0} = C_1 \cdot v_{\Lambda_1}, \\
e(\omega) & v_{\Lambda_1} = C_2 \cdot x_{\gamma_{\underline{2}}}(-1)x_{\gamma_2}(-1)v_{\Lambda_0}, \\
e(\omega) & v_{\Lambda_{\ell-1}} = C_3 \cdot x_{\gamma_{\underline{\ell}}}(-1) v_{\Lambda_{\ell}} \quad \textrm{and} \\
e(\omega) & v_{\Lambda_{\ell-1}} = C_4 \cdot x_{\gamma_{\ell}}(-1) v_{\Lambda_{\ell-1}}
\end{align}
where $C_i,  \ i = 1,2,3,4$ are some constants.
We will also use the operator
$$e(n \omega) \cong e(\omega)^{n}, \qquad \textrm{for $n \in \N$}.$$
Further on, on tensor product we can define the tensor product of $e(\lambda)$ operators:
\begin{align}
 e(\lambda) = e(\lambda) \otimes e(\lambda) \otimes \dots \otimes e(\lambda): V_P^{\otimes k} \longrightarrow V_P^{\otimes k}
\end{align}
and equation \eqref{omega} is again valid
for $\gamma \in \Gamma$.  
\section{Spinor representations of $\g$}\label{spinor}
We will revise now in short spinor representations of the finite dimensional Lie algebra $\g$ of type $D_{\ell}$ (see, for instance, \cite{FH} for details). Spinor representation are of fundamental weights $\omega_{\ell-1}$ and $\omega_{\ell}$ and can be realized via embedding of algebra $\g$ into Clifford algebra $\textrm{Cliff}(V, B) = \textrm{C}(B)$ where $V=\C ^{2 \ell}$ and $B$ is a symmetric, bilinear form on $V$. If we set $V = W \oplus W'$ where $W$ and $W'$ are the maximal isotropic subspaces for the form $B$, then we have an isomorphism
$$\textrm{C}(V,B) \cong \textrm{End} (\wedge^* W)$$
where $\wedge^* W$ denotes the exterior algebra of W: $\wedge^* W = \wedge^0 W \oplus \dots \oplus \wedge^ {\ell} W$. The embedding of $D_{\ell}$ into Clifford algebra gives therefore a representation of $D_{\ell}$ on $\wedge^*W$ and this representation decomposes into two irreducible ones. The following holds:
\begin{itemize}
 \item [-] if $\ell$ is even, then $\wedge^{(\textrm{even})} W$ and $\wedge^{(\textrm{odd})} W$ are the irreducible representations of weights $\omega_{\ell}$ and $\omega_{\ell-1}$ respectively,
\item [-] if $\ell$ is odd, then $\wedge^{(\textrm{even})} W$ and $\wedge^{(\textrm{odd})} W$ are the irreducible representations of weights $\omega_{\ell-1}$ and $\omega_{\ell}$ respectively.
\end{itemize}
Let $e_I = e_{i_1} \wedge \dots \wedge e_{i_k}, \ I=\{ i_1, \dots, i_k\} \subset \{ 1, \dots, \ell \}$ be a vector of the standard basis for $\wedge^* W$. Then $e_I$ is of the weight $\frac{1}{2} \left(\sum_{i \in I} \epsilon_i - \sum_{j \notin I} \epsilon_j \right)$ and spans the associated one dimensional weight subspace. These are at the same time the only weights of spinor representations. It is clear that $e_{ \{ 1,2, \dots, \ell-1 \} }$ represents a highest weight vector of weigh $\omega_{\ell-1} $, while $e_{ \{1,2, \dots, \ell-1, \ell \} }$ is a highest weight vector of weight $\omega_{\ell}$. 

In the reminder of the text we will denote the weight vectors with $e_{\Sigma} = w_{\Sigma}$ where $\Sigma \subset \{1, \dots, \ell \}$. The irreducible $\g$-modules of weight $\omega_{\ell-1}$ and $\omega_{\ell}$ are on the top of the $\gtl$-modules $L(\lam)$ and $L(\llam)$ modules with respect to the gradation \eqref{gradacija}. Using \eqref{formula} we can easily deduce that in the lattice construction weight vectors $w_{\Sigma}$ are of the form $1 \otimes e^{\mu} = e^{\mu}$ where $\mu$ denotes their weight, $\mu = \frac{1}{2} \left( \sum_{i \in \Sigma} \epsilon_i - \sum_{j \notin \Sigma} \epsilon_j \right)$. 

A simple calculation shows that spinor representations weights satisfy the commutativity condition \eqref{uvjet_kom} if and only if $1 \in \Sigma$, i.e. if
$$\mu = \frac{1}{2} \left(  \epsilon_1 + \dots \right)$$
Therefore vertex operators $\mathcal{Y}(1 \otimes e^{\mu},z) = \mathcal{Y}(w_{\Sigma},z)$ commute with the action of $\gtl_1$ if and only if $1 \in \varGamma$ and from now on we will assume $1 \in  \Sigma$ and replace $\Sigma$ with $\Gamma_1$. 

\section{Vertex operator relations}\label{uvjeti_razlike} In this section we find the relations between vector fields. We first observe the identities we have on vacuum vector, i.e. $v_{\Lambda_0}$. Using \eqref{formula} we see that
\begin{lemma} \label{vakuum0}
\begin{align} 
&  \label{vakuum1} x_{\delta}(-1)x_{\gamma}(-1) v_{\Lambda_0} = 0 \qquad \textrm{if} \quad \delta \neq \underline{\gamma} , \\
&  \label{vakuum2} x_{\underline{\gamma}}(-1) x_{\gamma}(-1) v_{\Lambda_0} = C \cdot x_{\underline{\delta}}(-1) x_{\delta}(-1) v_{\Lambda_0} = C' \cdot e^{2 \omega} v_{\Lambda_0} \quad \textrm{and} \\
& \label{vakuum3} x_{\tau}(-1) x_{\delta}(-1)x_{\gamma}(-1) v_{\Lambda_0} = 0
\end{align}
for $\gamma, \delta, \tau \in \Gamma.$ 
\end{lemma}
\begin{proof}
Let us show relation \eqref{vakuum1}. We first analyse   $\mathcal{Y}(e^{\gamma},z)v_{\Lambda_0}$. Since $\langle \gamma, \omega_0 \rangle = 0$, formula \eqref{formula} implies that minimal $j$ such that $x_{\gamma}(-j)v_{\Lambda_0} \neq 0$ is $j=1$. We get:
\begin{align}
x_{\gamma}(-1) v_{\Lambda_0} = C \cdot e^{\gamma + \omega_0} = C \cdot e^{\gamma}.
\end{align}
Now we observe $\mathcal{Y}(e^{\delta},z)e^{\gamma}$ for $\delta \neq \underline{\gamma}$. Since $\langle \delta, \gamma \rangle \geq 1$, minimal $j$ such that $x_{\delta}(-j)e^{\gamma} \neq 0$ is $j=2$. Therefore we have $x_{\delta}(-1)e^{\gamma} = 0$, i.e. $x_{\delta}(-1)x_{\gamma}(-1) v_{\Lambda_0} = 0 $.

If, on the other hand, $\delta = \underline{\gamma}$, then $\langle \underline{\gamma}, \gamma \rangle = 0$ and minimal $j$ such that $x_{\underline{\gamma}}(-j)e^{\gamma} \neq 0$ is $j=1$. We have
$$x_{\underline{\gamma}}(-1)x_{\gamma}(-1) v_{\Lambda_0} = C \cdot x_{\underline{\gamma}}(-1)e^{\gamma} = C' \cdot e^{\gamma + \underline{\gamma}} = C' \cdot e^{2 \omega} v_{\Lambda_0} $$
i.e. we got the relation \eqref{vakuum2}. Since $x_{\delta}(-1)x_{\gamma}(-1) v_{\Lambda_0} = 0$ if $\delta \neq \underline{\gamma}$, the only claim left to prove is the relation 
$$x_{\tau}(-1) x_{\underline{\gamma}}(-1) x_{\gamma}(-1) v_{\Lambda_0} = 0.$$
Formula \eqref{formula} implies again that minimal $j$ such that $x_{\tau}(-j) x_{\underline{\gamma}}(-1) x_{\gamma}(-1) v_{\Lambda_0} \neq 0$ must be $j=2$ because $\langle \tau, 2 \omega \rangle = 2$. This completes the proof.
\end{proof}

Since algebra $\gtl_1$ is commutative, for $\gamma, \delta \in \Gamma$ vertex operator $Y(x_{\gamma}(-1)x_{\delta}(-1)v_{\Lambda_0},z)$ equals to the product of $x_{\gamma}(z)$ and $x_{\delta}(z)$ as ordinary Laurent series (cf. \cite{DL}). In this way we get the relations between vertex operators on level 1 modules.
\begin{theorem} Let $\gamma, \delta \in \Gamma$. We have:
\begin{eqnarray}\label{relacije}
& & \label{rel1} x_{\delta}(z)x_{\gamma}(z)  = 0 \qquad \textrm{if} \quad \delta \neq \underline{\gamma}, \\
& & \label{rel2} x_{\underline{\gamma}}(z) x_{\gamma}(z)  = C \cdot x_{\underline{\delta}}(z) x_{\delta}(z) \qquad \textrm{and}\\
& & \label{rel3} x_{\tau}(z) x_{\underline{\gamma}}(z) x_{\gamma}(z)  = 0.
\end{eqnarray}
\end{theorem}
\section{Difference and initial conditions for level 1 modules}\label{pocetak1}
Now, using relations between fields from the previous section, we reduce the spanning set \eqref{PBW} of the Feigin-Stoyanovsky's type subspaces of level 1 modules, i.e. $L(\Lambda_0)$, $L(\Lambda_1)$, $L(\Lambda_{\lam})$ and $L(\llam)$. We further observe which factors of $-1$ degree, i.e. $x_{\gamma}(-1)$ factors, annihilate the highest weight vectors or can be expressed as a sum of greater monomials on this vectors. 
\subsection{Difference conditions}
The coefficients of $z^{n-2}$ in the equalities \eqref{rel1} and \eqref{rel2} give us the relations between the infinite sums of monomials from $U(\gtl_1)$:
\begin{align}
0 &= x_{\delta}(z)x_{\gamma}(z) = \sum_n \left( \sum_{p+s=n}  x_{\delta}(-p)x_{\gamma}(-s) \right) z^{n-2} = \label{rell1} \\ &= \sum_n \left( \sum_p x_{\delta}(-n+p) x_{\gamma}(-p) \right) z^{n-2} \notag 
\end{align}
and
\begin{align}
0 &= x_{\underline{\gamma}}(z)x_{\gamma}(z) - C \cdot x_{\underline{\delta}}(z)x_{\delta}(z)  \label{rell2} = \\ &= \sum_n \left( \sum_{p} (x_{\underline{\gamma}}(-n+p)x_{\gamma}(-p) - C \cdot x_{\underline{\delta}}(-n+p)x_{\delta}(-p) ) \right) z^{n-2} \notag. 
\end{align}

Therefore 
\begin{align}
& \label{rell3} \sum_p x_{\delta}(-n+p)x_{\gamma}(-p)=0 \qquad \textrm{and} \\ & \sum_{p} (x_{\underline{\gamma}}(-n+p)x_{\gamma}(-p) - C \cdot x_{\underline{\delta}}(-n+p)x_{\delta}(-p) ) =0
\end{align}
for every $n$. The minimal monomials of this equalities, with regard to order introduced in Section \ref{order}, can now be expressed as the sum of other terms and so excluded from the spanning set. Such minimal monomials are called \textit{the leading terms}. 

For instance, relation \eqref{rell3} gives us the following leading terms for $\delta \neq \underline{\gamma}$:
\begin{align}
& x_{\delta}(-j)x_{\gamma}(-j)  \qquad \textrm{if $n$ is even and} \\
& x_{\delta}(-j-1)x_{\gamma}(-j)  \qquad \textrm{if $\delta \geq  \gamma$ and $n$ is odd}. 
\end{align}
This is clear when we remember that we first compare shapes and then colors of the monomials.

If the monomial $x(\pi)$ contains a leading term, this leading term can be replaced with an appropriate sum of greater monomials and so, using Proposition \ref{uredjaj}, it follows that $x(\pi)$ itself can be expressed as the sum of greater monomials. Therefore we conclude that monomials which contain the leading terms can be excluded from the spanning set. Monomials which don't contain the leading terms are said to satisfy \textit{the difference conditions} (or $DC$ in short) for $W(\Lambda)$.

These monomials can be described as:
\[ x_{\delta}(-i) x_{\gamma}(-j) \ \textrm{satisfies $DC$ if} \ \left\{ \begin{array}{lll} -i \leq -j-2, & \\ -i= -j-1 & \mbox{and $\delta < \gamma$ or} \\ & \mbox{$\delta = \gamma_{\ell}$ and $\gamma=\gamma_{\underline{\ell}},$} \\ -i=-j & \mbox{and $\delta = \gamma_{\underline{2}}$, $\gamma=\gamma_2$.} \end{array} \right. \]

Let $b_r$ (resp. $b_{\underline{r}}$) be the number of $x_{\gamma_r}(-j-1)$ (resp. $x_{\gamma_{\underline{r}}}(-j-1)$), and $a_r$ (resp. $a_{\underline{r}}$) of $x_{\gamma_r}(-j)$ (resp. $x_{\gamma_{\underline{r}}}(-j)$) factors in monomial $x(\pi)$. Then the difference conditions can be also written as:
\begin{itemize}
\item[1)] $b_r + b_{r-1} + \dots + b_2 + a_{\underline{2}} + \dots
+ a_{\underline{\ell}} +  a_\ell + \dots + a_{r+1}     \leq 1, \quad r =
2, \dots, \ell-1$, \item[2)] $b_r + \dots + b_2 + a_{\underline{2}} +
\dots+a_{\underline{r-1}} + a_{\underline{r+1}} \dots+
a_{\underline{\ell}} +  a_\ell + \dots + a_r \leq 1$,
\quad $r = 2, \dots, \ell-1$,
\item[3)] $b_\ell + b_{\ell-1} + \dots + b_2 + a_{\underline{2}} + \dots
+ a_{\underline{\ell-1}} +  a_\ell \leq 1$, \item[4)] $b_{\underline{\ell}}
+ b_{l-1} + \dots + b_2 + a_{\underline{2}} + \dots +
a_{\underline{\ell-1}} +  a_{\underline{\ell}} \leq 1$, \item[5)]
$b_{\underline{r+1}} + \dots + b_{\underline{\ell}} + b_{\ell} + \dots +
b_2 + a_{\underline{2}} + \dots + a_{\underline{r}} \leq 1$,
\quad $r = \ell-1, \dots, 2$,
\item[6)] $b_{\underline{r}} + b_{\underline{r+1}} \dots +
 + b_{r+1} + b_{r-1} + \dots +
b_2 + a_{\underline{2}} + \dots + a_{\underline{r}} \leq 1$,
\quad $r = \ell-1, \dots, 2$.
\end{itemize}
The proof is straightforward. 

\subsection{Initial conditions}\label{pocetni_uvjeti} We further reduce the spanning set by observing that certain factors $x_{\delta}(-1)$ annihilate some highest weight vectors and as such can not be contained in the monomial basis. Further on, some monomials of the shape $(-1)(-1)$ can be expressed via greater monomials on highest weight vector $v_{\Lambda_0}$. The following lemma describes this cases:
\begin{lemma} \label{poc_uvj_pom} Let $v_{\Lambda_i}$ be the highest weight vector of module $L(\Lambda_i)$, $i=0,1,\ell-1, \ell$. We have:
\begin{itemize}
 \item[a)] $x_{\underline{\gamma}}(-1)x_{\gamma}(-1)v_{\Lambda_0} = C \cdot x_{\gamma_{\underline{2}}}(-1)x_{\gamma_2}(-1)v_{\Lambda_0} \quad$ for $\gamma \in \Gamma$,
\item[b)] $x_{\tau}(-1)x_{\delta}(-1)x_{\gamma}(-1)v_{\Lambda_0}=0 \quad$ for $\tau, \delta, \gamma \in \Gamma$,
\item[c)] $x_{\gamma}(-1)v_{\Lambda_{\ell-1}} = 0 \quad$ if and only if $\gamma \in \{\gamma_2, \dots, \gamma_{\ell-1}, \gamma_{\underline{\ell}} \}$,
\item[d)] $x_{\gamma}(-1)v_{\Lambda_{\ell}} = 0 \quad$ if and only if $\gamma \in \{\gamma_2, \dots, \gamma_{\ell-1}, \gamma_{\ell} \}$,
\item[e)] $x_{\gamma}(-1)v_{\Lambda_1}=0 \quad$ for $\gamma \in \Gamma$.
\end{itemize}
\end{lemma}
\begin{proof}
Claims a) and b) have been proved in Section \ref{uvjeti_razlike}. We show c) using formula \eqref{formula}. Let $\gamma \in \{\gamma_2, \dots, \gamma_{\ell-1}, \gamma_{\underline{\ell}} \}$. Then $\langle \gamma, \omega_{\ell-1} \rangle = 1$ so the minimal $j$ such that $x_{\gamma}(-j)v_{\Lambda_{\ell-1}} \neq 0$ is $j=2$. This implies $x_{\gamma}(-1)v_{\lam}=0$. On the other hand, for $ \gamma \in \{ \gamma_{\ell}, \gamma_{\underline{\ell-1}}, \dots, \gamma_{\underline{2}} \}$ we have $\langle \gamma, \omega_{\ell-1} \rangle = 0$ so $x_{\gamma}(-1)v_{\lam} \neq 0$. This proves c) and d) is similar. It remains to show e). For all $\gamma \in \Gamma$, we have $\langle \gamma, \omega \rangle =1$ so again the minimal $j$ such that $x_{\gamma}(-j)v_{\Lambda_1} \neq 0$ must be $j=2$. Therefore $x_{\gamma}(-1)v_{\Lambda_1} = 0$ as claimed.
\end{proof}
We will say that monomial $x(\pi) \in S(\gtl^{-}_1)$ satisfies \textit{the initial conditions} for $W(\Lambda)$ (or on $v_{\Lambda}$) if it contains neither the $x_{\gamma}(-1)$ elements which annihilate $v_{\Lambda}$ nor the factors of the $(-1)(-1)$ shape which can be expressed as a sum of greater monomials on $v_{\Lambda}$.

Lemma \ref{poc_uvj_pom} gives us the initial conditions for Feigin-Stoyanovsky's subspaces of level 1 modules:
\begin{itemize}
\item[1)] for $\Lambda = \Lambda_0$, a monomial $x(\pi)$ satisfies the initial conditions if it either doesn't contain $x_{\gamma}(-1)$ elements or is of the form: 
\begin{itemize}
\item[1a)] $x(\pi) = \dots (-j)x_{\gamma}(-1)$, \ $j \geq 2$, or  
\item[1b)] $x(\pi) = \dots (-j)x_{\gamma_{\underline{2}}}(-1) x_{\gamma_2}(-1)$, \ $j \geq 2$,
\end{itemize}

 \item[2)] for $\Lambda = \lam$, a monomial $x(\pi)$ satisfies the initial conditions if it either doesn't contain $x_{\gamma}(-1)$ elements or is of the form: \\
$x(\pi) = \dots (-j)x_{\gamma}(-1)$, \ $j \geq 2$  for  $\gamma \in \{\gamma_{\ell}, \gamma_{\underline{\ell-1}}, \dots, \gamma_{\underline{2}} \}$,

\item[3)] for $\Lambda = \llam$, a monomial $x(\pi)$ satisfies the initial conditions if it either doesn't contain $x_{\gamma}(-1)$ elements or is of the form: \\
$x(\pi) = \dots (-j)x_{\gamma}(-1)$, \ $j \geq 2$  for  $\gamma \in \{\gamma_{\underline{\ell}}, \gamma_{\underline{\ell-1}}, \dots, \gamma_{\underline{2}} \}$,

\item[4)] for $\Lambda = \Lambda_1$, a monomial $x(\pi)$ satisfies the initial conditions if it doesn't contain $x_{\gamma}(-1)$ elements.

\end{itemize}
Those conditions can be also expressed using inequalities. Let $b_{i}$ denote the number of $x_{\gamma_i}(-1)$ factors in the monomial $x(\pi)$ and $b_{\underline{i}}$ the number of $x_{\gamma_{\underline{i}}}(-1)$ factors. Then $x(\pi)$ satisfies the initial conditions if:
\begin{itemize}
 \item[-] on $v_{\Lambda_0}$:

\begin{itemize}

\item[1)] $b_{\underline{3}} \dots + b_{\underline{\ell}} + b_{\ell} +
\dots + b_2  \leq 1$,

\item[2)] $b_{\underline{2}} + b_{\underline{3}} \dots +
b_{\underline{\ell}} + b_{\ell} + \dots + b_{3}  \leq 1$,

\end{itemize}

\item[-] on $v_{\Lambda_{\ell-1}}$:

\begin{itemize}
\item[1)] $b_{\underline{\ell}} + b_{\ell-1} + b_{\ell-2} + \dots + b_2 \leq 0 $,  \item[2)] $b_\ell + b_{\ell-1} +
\dots + b_2  \leq 1$, \item[3)] $b_{\underline{2}} + \dots +
b_{\underline{\ell}} + b_{\ell} + \dots + b_{3}  \leq 1$,
\end{itemize}

\item[-] on $v_{\Lambda_l}$:
\begin{itemize}
\item[1)] $b_\ell + b_{\ell-1} + b_{\ell-2} + \dots + b_2 \leq 0 $, \item[2)] $b_{\underline{\ell}} + b_{\ell-1} +
\dots + b_2  \leq 1$, \item[3)] $b_{\underline{2}} + \dots +
b_{\underline{\ell}} + b_{\ell} + \dots + b_{3}  \leq 1$,
\end{itemize}

\item[-] on $v_{\Lambda_1}$:

\begin{itemize}

\item[1)] $b_{\underline{3}} \dots + b_{\underline{\ell}} + b_{\ell} +
\dots + b_2   \leq 0$,
\item[2)] $b_{\underline{2}} +  \dots + b_{\underline{\ell}} + b_{\ell}
+ \dots + b_{3}  \leq 0$.
\end{itemize}
\end{itemize}
\begin{remark} \label{dc_ic}
We can memorize the initial conditions via the difference conditions if we add the so-called imaginary factors of degree zero. Then the $(-1)$ elements that can follow by difference conditions are exactly the ones allowed by the initial conditions. This will help us to define initial conditions for $W(\Lambda)$ for the level 2 modules. The rules for the adding of the imaginary factors of degree zero are:
\begin{itemize}
 \item[-] if $v_{\Lambda} = v_{\Lambda_0}$ then we don't add anything,
 \item[-] if $v_{\Lambda} = v_{\lam}$, we add $x_{\gamma_{\underline{\ell}}}(-0)$,
\item[-]if $v_{\Lambda} = v_{\llam}$, we add $x_{\gamma_{\ell}}(-0)$,
\item[-] if $v_{\Lambda} = v_{\Lambda_1}$, we add $x_{\gamma_{\underline{2}}}(-0) x_{\gamma_2}(-0)$.
\end{itemize}
\end{remark}
Now from Proposition \eqref{uredjaj} we get the following proposition:
\begin{proposition}\label{razap}
The set
\begin{align}
\label{razap1} \{ \ x(\pi)v_{\Lambda} \ \vert \ x(\pi) \ \textrm{satisfies DC and IC conditions for $W(\Lambda)$}  \ \}
\end{align}
spans $W(\Lambda)$.
\end{proposition}

We will call the vectors which satisfy both the $DC$ and the $IC$ for $W(\Lambda)$ t\textit{he admissible vectors for $W(\Lambda)$.} 

\section{Basic $\gtl^{-}_1$-intertwining operators}\label{operatori_1}
We prove the linear independence of the vectors from Proposition \ref{razap} using the  $\gtl^{-}_1$-intertwining operators. First we will define so-called basic operators and their compositions will then give the operators we require. 

In Section \ref{spinor} we saw that the vertex operators $\mathcal{Y}(1 \otimes e^{\mu}, z) = \mathcal{Y}(w_{\Gamma_1}, z)$ commute with the action of $\gtl_1$ because $1 \in \Gamma_1$. We will analyse the action of certain coefficients of this operator on highest weigh vectors. 
\begin{lemma}
Let $w_{\Gamma_1} = 1 \otimes e^{\mu}$ be a weight vector of the spinor represenations on the top of the $\gtl$-moldules $L(\lam)$ and $L(\llam)$. For $n \geq 0$ there exists a coefficient of the intertwining operator  $\mathcal{Y}(w_{\Gamma_1},z) = \mathcal{Y}(1 \otimes e^{\mu},z)$, let's denote it as $w(n, \Gamma_1)$, such that 
\begin{align} \label{op1}
w(n, \Gamma_1) e^{n \omega}v_{\Lambda_0} = C_n \cdot e^{n \omega} w_{\Gamma_1}, \quad n \geq 0.
\end{align}
We call such coefficients the basic intertwining operators of the first group.
\end{lemma}
\begin{proof}
Formula \eqref{formula} shows that the coefficient of $z^{\langle \mu, n \omega + \omega_0 \rangle } = z^{\frac{n}{2}}$ maps $e^{n\omega}v_{\Lambda_0}$ into $D_n \cdot e^{\mu + n \omega + \omega_0} = C_n \cdot e^{n \omega} w_{\Gamma_1}$ for $n \geq 0$.
\end{proof}

Now we define a notion of a complementary weight. If $\mu$ is the weight of the weight vector $w_{\Gamma_1}$, then its \textit{complementary weight $\mu'$} is the weight of vector $w_{\Gamma_1'}$ where $\Gamma_1' = \{1,2, \dots, \ell \ \}- \Gamma_1 \cup \{ 1\}$. In other words, we get the weight $\mu'$ by diverting all the signs of $\epsilon_i$ in $\mu$ except $\epsilon_1$ which remains of the positive sign. This implies immediately that $\mathcal{Y}(1 \otimes e^{\mu'}, z)$ commutes with the action of the $\gtl_1$ as well. We have the following lemma:

\begin{lemma}
Let $n	\geq 0$. Then the coefficient of $z^p$ (denote it with $w(n,\Gamma_1')$), where $p = \langle \mu', \mu + n \omega \rangle$, sends $e^{n \omega} w_{\Gamma_1}$ to $e^{n \omega} v_{\Lambda_1}$, i.e.
\begin{align}\label{op2}
w'(n,\Gamma_1')  e^{n \omega}w_{\Gamma_1} = C_n' \cdot e^{n \omega} v_{\Lambda_1}, \quad n \geq 0.
\end{align}
We call this $w'(n, \Gamma_1')$ coefficients the basic intertwining operators of the second group.
\end{lemma}
\begin{proof}
Directly from \eqref{formula}.
\end{proof}
The operators from the second group, i.e. operators \eqref{op2}, have the following important property:
\begin{lemma}\label{svojstvo_op}
Let $w'(n, \Gamma_1^{(1)} \hspace{0.3mm} ' )$ be the basic intertwining operator of the second group such that $w'(n, \Gamma_1^{(1)} \hspace{0.3mm} ') e^{n \omega} w_{\Gamma_1^{(1)}} = C \cdot e^{n \omega} v_{\Lambda_1}$. Then $w'(n, \Gamma_1^{(1)} \hspace{0.3mm} ') e^{n \omega} w_{\Gamma_1^{(2)}} = 0$ if $\Gamma_1^{(1)} \neq \Gamma_1^{(2)}$. 
\end{lemma}
\begin{proof}
Again by applying \eqref{formula}.
\end{proof}
\begin{remark} \label{grupe34}
Operators from \eqref{op1} and \eqref{op2} can be composed; we take the basic operator from the first group which sends $e^{m \omega} v_{\Lambda_0}$ to $C \cdot e^{m \omega} w_{\Gamma_1}$ and compose it with the basic operator from the second group which maps $e^{m \omega} w_{\Gamma_1}$ to $C' \cdot e^{m \omega} v_{\Lambda_1}$. In this way we obtain an intertwining operator, denote it with $w(m,m,0)$, such that 
\begin{align} \label{op3}
w(m,m,0)e^{m\omega}v_{\Lambda_0} = C'' \cdot e^{m \omega} v_{\Lambda_1}.
\end{align}
Let now $n > m$. Composing the operators from above we can now get the following map:
\begin{align*}
& e^{m \omega}v_{\Lambda_0} \overset{w(m,m,0)}{\mapsto} C_1 \cdot e^{m \omega} v_{\Lambda_1} = C_1' \cdot e^{(m+1) \omega} v_{\Lambda_0}  \overset{w(m+1,m+1,0)}{\mapsto} D_1 \cdot e^{(m+1)} v_{\Lambda_1} = \\ & = D_1' \cdot e^{(m+2)} v_{\Lambda_0} \overset{w(m+2,m+2,0)}{\mapsto} \dots \overset{w(n,n,0)}{\mapsto} D_{n-1} \cdot e^{(n-1) \omega} v_{\Lambda_1} = D_{n-1}' \cdot e^{n \omega} v_{\Lambda_0}. 
\end{align*}
So, for $n >m$ we obtained an intertwining operator, let's denote it with $w(m,n,0)$, such that
\begin{align} \label{op4}
w(m,n,0)e^{m\omega}v_{\Lambda_0} = C \cdot e^{n \omega} v_{\Lambda_0}.
\end{align}
This will be \textit{the basic operators of the third group}. 
\end{remark}

In further text we will need the following technical lemma: 

\begin{lemma}\label{lem_op_1} \

\begin{itemize} 
\item[i)] Let $\delta = \gamma_{\underline{j}} = \epsilon_1 - \epsilon_j $. Then $x_{\delta}(-1)w_{\Gamma_1} \neq 0$ if and only if $j \in \Gamma_1$. In this case, $x_{\delta}(-1)w_{\Gamma_1} = C \cdot e^{\omega} w_{\Gamma_1'}$ where $\Gamma_1' = \Gamma_1 - \{j \}$. 
\item[ii)] If $\delta = \gamma_j = \epsilon_1 + \epsilon_j $, then $x_{\delta}(-1)w_{\Gamma_1} \neq 0$ if and only if $j \notin \Gamma_1$ and then we have $x_{\delta}(-1)w_{\Gamma_1} =C \cdot e^{\omega} w_{\Gamma_1'}$ where $\Gamma_1' = \Gamma_1 \cup \{j \}$. 
\item[iii)] $x_{\tau}(-1) x_{\delta}(-1) w_{\Gamma_1} = 0$ for all $\tau, \gamma \in \Gamma$.
\end{itemize}
\end{lemma}

\begin{proof}
Straightforward, using \eqref{formula}.
\end{proof}

We will call an array in a monomial of the shape $(-p) \dots (-j-2) (-j-1) (-j)$ (i.e. an array of factors whose degrees decrease by one) \textit{a successive block}. 

\begin{remark} \label{uzastopna} The above lemma can repeatedly be applied on such successive blocks. Namely, vectors $e^{\omega}v$ and $e(\omega)v$ (for $v = v_{\Lambda_0}, w_{\Gamma_1}$ or $v_{\Lambda_1}$) are collinear. In Section \ref{op_pr_struje} we saw that $e(\omega)$ lifts the degree of monomial factors by one when commuted to the left. So, we have $\dots x_{\delta}(-2) e^{\omega} w_{\Gamma_1} = \dots x_{\delta}(-2) \cdot C \cdot e(\omega) w_{\Gamma_1} $ and commuting $e(\omega)$ to the left gives $\dots C \cdot e(\omega) x_{\delta}(-1)w_{\Gamma_1}$ which enables us to apply Lema \ref{lem_op_1} again. We observe that the positive colors (see Section \ref{order}) in the successive block add the indices to $\Gamma_1$ (i.e. factor $x_{\gamma_i}(-j)$ from the successive block will add index $i$ to $\Gamma_1$). On the other hand, negative colors subtract the indices (i.e. factor $x_{\gamma_{\underline{i}}}(-j)$ will subtract index $i$ from $\Gamma_1$). The example below illustrates the principle.
\end{remark}

\begin{example}
\begin{align*}
& x_{\underline{2}}(-3)x_{\underline{3}}(-2)x_{\gamma_6}(-1)w_{12347} = x_{\underline{2}}(-3)x_{\underline{3}}(-2) \cdot C_1' \cdot e^{\omega}w_{123467} = \\
& x_{\underline{2}}(-3)x_{\underline{3}}(-2) \cdot C_1 \cdot e(\omega)w_{123467} = C_1 \cdot e(\omega)x_{\underline{2}}(-2)x_{\underline{3}}(-1) w_{123467} = \\
& C_1 \cdot e(\omega)x_{\underline{2}}(-2) \cdot C_2' e^{\omega} w_{12467} = \dots = C_1 C_2 C_3 w_{1267}.
\end{align*}
\end{example}

We have the following lemma:
\begin{lemma}\label{lem_op_2}
Let $x(\pi)$ be a successive monomial of the shape $(-j) \dots (-2)(-1)$ which satisfies $DC$ and $IC$ for $W(\llam)$ and let $x(\pi')$ be a different monomial of the same shape such that $x(\pi') > x(\pi)$. Then $x(\pi)v_{\llam} \neq C \cdot x(\pi')v_{\llam}$, i.e. those two vectors are not proportional. The same is true for $W(\lam)$.
\end{lemma}
\begin{proof}
We will outline the proof of the lemma for $v_{\llam} = w_{123\dots \ell}$. Let $x(\pi) = x_{\delta_j}(-j) \dots x_{\delta_1}(-1)$, $x(\pi') = x_{\tau_j}(-j) \dots x_{\tau_1}(-1)$ and let $(-i)$ be the first index from right to left where $\delta_i \neq \tau_i$. Since $x(\pi') > x(\pi)$ then $\tau_i > \delta_i$. We have the following cases:
\begin{itemize}
 \item [a)] $\delta_i = \gamma_{\ell}.$ Monomial $x(\pi)$ satisfies $DC$ and $IC$ for $W(\llam)$ so the colors $\delta_{p}, \ p \in \{1, \dots, i-1 \}$ must be from the set $\{ \gamma_{\underline{\ell}}, \gamma_{\ell} \}$, i.e. we have (by repeated use of Lemma \eqref{lem_op_1}):
\begin{align*}
& x_{\delta_i}(-j) \dots x_{\delta_1}(-1)v_{\llam} = \\ & = x_{\gamma_{\ell}}(-i)x_{\gamma_{\underline{\ell}}}(-(i-1)) \dots x_{\gamma_{\ell}}(-2)x_{\gamma_{\underline{\ell}}}(-1)v_{\llam} = \\ & = x_{\gamma_{\ell}}(-i) \cdot C \cdot e^{(i-1) \omega}v_{\lam}
\end{align*}
Now, since $\tau_i > \delta_i$, it follows that $\tau_{i} \in \{ \gamma_2, \dots, \gamma_{\ell-1} \}$ and therefore $$x(\pi')v_{\llam}= \dots x_{\tau_i}(-i) \cdot C \cdot e^{(i-1) \omega}v_{\lam} = 0$$ again by Lemma \ref{lem_op_1} .
\item[b)] $\delta_i = \gamma_{\underline{\ell}}$. Similarly like in a) we prove that $x(\pi')v_{\llam}=0$.
\item[c)] $\delta_i \in \{ \gamma_{\underline{\ell-1}}, \dots, \gamma_{\underline{2}} \}$. By applying Lemma \ref{lem_op_1}, we first notice that vector $x_{\delta_{i-1}}(-(i-1))\dots x_{\delta_1}(-1) v_{\llam}$ equals to $C \cdot e^{(i-1) \omega} w_{\Gamma_1}$. Set $\Gamma_1$ is a subset of the set $\{1, 2, 3 , \dots, \ell \}$, i.e. we get $\Gamma_1$ by substracting some indices from $\{1, 2, 3 , \dots, \ell \}$. This indices correspond to the negative colors in the starting block $x_{\delta_{i-1}}(-(i-1)) \dots x_{\delta_1}(-1)$ of the monomial $x(\pi)$. 

Since two monomials $x(\pi)$ and $x(\pi')$ coincide up to $x_{\delta_i}(-i)$ factor, we also have $x(\pi')v_{\llam} = \dots x_{\tau_i} (-i) \cdot C \cdot e^{(i-1)\omega} w_{\Gamma_1}$. Now we distinguish two subcases:
\begin{itemize}
\item[c1)] $\tau_i = \gamma_j$ for some $j \in \{2, \dots, \ell \}$. If $x(\pi') v_{\llam} \neq 0$, then (according to Lemma \ref{lem_op_1}) $j \neq \Gamma_1$. Let $x(\pi)v_{\llam} = x_{\delta_m}(-m) \dots x_{\delta_i}(-i) \cdot C \cdot e^{(i-1)\omega} w_{\Gamma_1} = C' \cdot e^{m\omega} w_{\Gamma_1'}$. Because $\delta_i \in \{ \gamma_{\underline{\ell-1}}, \dots, \gamma_{\underline{2}} \}$ and $x(\pi)$ satisfies the DC, we have $\delta_s \in \{ \gamma_{\underline{\ell-1}}, \dots, \gamma_{\underline{2}} \}$ for $i \leq s \leq m$. Therefore we obtain the set $\Gamma_1'$ by subtracting $m-(i-1)$ different elements from $\Gamma_1$. If we now denote $x(\pi')v_{\llam} =x_{\tau_m}(-m) \dots x_{\tau_i} (-i) \cdot C \cdot e^{(i-1)\omega} w_{\Gamma_1} = C'' \cdot e^{m \omega} w_{\Gamma_1''}$, then $\Gamma_1''$ has more elements than $\Gamma_1'$ because $x_{\tau_i}(-i)$ factor added index $j$ in $\Gamma_1$.
\item[c2)] $\tau_i = \gamma_ {\underline{j}}$. Similarly as in c1), if $x(\pi')v_{\llam} \neq 0$, we prove $\Gamma_1 ' \neq \Gamma_1''$.
\end{itemize}
\end{itemize}
\end{proof}
For a highest weigh vector $v_{\Lambda}$ and each successive block \\ $x(\pi)= x_{\delta_j}(-j) \dots x_{\delta_2}(-2) x_{\delta_1}(-1)$, $j \in \N$, which satisfies $DC$ and $IC$ for $W(\Lambda)$ we will define a $\gtl_1^{-}$-intertwining operator $I_{\pi,\Lambda}$. We will call this operators \textit{the basic intertwinig operators} and we will use them in the proof of linear independence of admissible vectors. We have the following cases:
\begin{itemize}
 \item [a)]  $x(\pi) =  x_{\gamma_{i_j}}(-j) \dots x_{\gamma_{i_1}}(-1)$ where $\gamma_{i_s} \in \{ \gamma_2, \dots, \gamma_{\ell-1} \}$, $s=1, \dots, j$. Because of the initial conditions (see \ref{pocetni_uvjeti}), such monomial annihilates all highest weigh vectors except $v_{\Lambda_0}$. Now we define the following three operators for $x(\pi)v_{\Lambda_0}$ (see \eqref{op1}):
\begin{align} 
 \label{operator11} I^1_{\pi, \Lambda_0} & = w(0, \Gamma_1) \quad \textrm{i.e.} \\  \notag I^1_{\pi, \Lambda_0}v_{\Lambda_0} &= C \cdot w_{\Gamma_1}   \quad \textrm{where} \quad \Gamma_1 = \{1, 2, 3, \dots, \ell-1 \}  - \{i_1, \dots, i_j \},  \\ \label{operator12} I^2_{\pi, \Lambda_0} & =  w(0, \Gamma_1) \quad \textrm{i.e.} \\ \notag I^2_{\pi, \Lambda_0}v_{\Lambda_0} &=C \cdot w_{\Gamma_1} \quad \textrm{where} \quad \Gamma_1 = \{1,2, 3, \dots, \ell-1, \ell \} - \{ i_1, \dots, i_j \},  \\ \label{operator13} I^3_{\pi, \Lambda_0} & =  w(0, \Gamma_1) \quad \textrm{i.e.} \\ \notag I^3_{\pi, \Lambda_0}v_{\Lambda_0} &=C \cdot w_{\Gamma_1} \quad \textrm{where} \quad \Gamma_1 = \{1,2, 3, \dots,  i_j \} - \{ i_1, \dots, i_j \}.
\end{align}
Then we have $I^1_{\pi, \Lambda_0}x(\pi)v_{\Lambda_0} =  C \cdot e^{j \omega} v_{\lam}$, $I^2_{\pi, \Lambda_0}x(\pi)v_{\Lambda_0} =  C \cdot e^{j \omega} v_{\llam}$ and $I^3_{\pi, \Lambda_0}x(\pi)v_{\Lambda_0} =  C \cdot e^{j \omega} w_{123 \dots i_j}$. 
\item[b)] $x(\pi) =  x_{\delta_j}(-j) \dots x_{\delta_1}(-1)$ where $\delta_1 \leq \gamma_{\underline{\ell}}$. Then it follows from $DC$ that $\delta_p \in \{\gamma_{\ell}, \gamma_{\underline{\ell}}, \dots, \gamma_{\underline{2}} \}$ for $p \in \{2, \dots, j \}$. Such monomials don't annihilate $v_{\Lambda_0}$ and $v_{\llam}$. We first define $I_{\pi,\Lambda_0}$. This operator will be composition of the two following operators: 
\\
\noindent i) operator $I_{\pi1, \Lambda_0} = w(0, \{12 \dots \ell \})$ which maps $v_{\Lambda_0} \mapsto C \cdot v_{\llam}$ (see \eqref{op1}). \\
ii) operator $I_{\pi2, \Gamma_1} = w'(j,\Gamma_1')$ which belongs to the basic operators of the second group, i.e. to the operators which map $e^{j \omega} w_{\Gamma_1} \mapsto C \cdot e^{j \omega} v_{\Lambda_1}$ for some $j \geq 0$ and some $\Gamma_1 \subset \{1,2, \dots, \ell \}$ (see \eqref{op2}). Namely, after applying operator $I_{\pi1, \Lambda_0}$ we get: 
\begin{align*}
& I_{\pi1, \Lambda_0} (x_{\delta_j}(-j) \dots x_{\delta_1}(-1) v_{\Lambda_0})  = x_{\delta_j}(-j) \dots x_{\delta_1}(-1) I_{\pi_1, \Lambda_0} v_{\Lambda_0} = \\ & = C \cdot x_{\delta_j}(-j) \dots x_{\delta_1}(-1) v_{\llam} = C' \cdot e^{j \omega} w_{\Gamma_1}
\end{align*}
by Lemma \ref{lem_op_1}. Now $I_{\pi2, \Gamma_1}$ sends this $e^{j \omega} w_{\Gamma_1}$ to $e^{j \omega} v_{\Lambda_1}$. \\
\noindent The final $I_{\pi, \Lambda_0}$ for $x(\pi)v_{\Lambda_0}$ is then defined as $I_{\pi2, \Gamma_1} \circ I_{\pi1,\Lambda_0}$.
\vspace{2mm}

\noindent It remains to define $I_{\pi, \llam}$. We set simply $I_{\pi,\llam} = I_{\pi_2, \Gamma_1}$ since we have:
$$x(\pi)v_{\llam}  = x_{\delta_j}(-j) \dots x_{\delta_1}(-1) v_{\llam}  = C \cdot e^{j \omega}w_{\Gamma_1}.$$
\item[c)] $x(\pi) =  x_{\delta_j}(-j) \dots x_{\delta_2}(-2) x_{\gamma_{\ell}}(-1)$. Then by $DC$ $\delta_p \in \{\gamma_{\ell}, \gamma_{\underline{\ell}}, \dots, \gamma_{\underline{2}} \}$ for $p \in \{2, \dots, j \}$ and by $IC$ such monomial doesn't annihilate $v_{\Lambda_0}$ and $v_{\lam}$ (see \ref{pocetni_uvjeti}). We proceed as in b) and define first $I_{\pi, \Lambda_0}$ as $I_{\pi_2, \Gamma_1} \circ I_{\pi_1, \Lambda_0}$. Here $I_{\pi_1, \Lambda_0}v_{\Lambda_0} = C \cdot v_{\lam}$. Since we have
$$x(\pi)v_{\lam} =  x_{\delta_j}(-j) \dots x_{\delta_2}(-2) x_{\gamma_{\ell}}(-1) v_{\lam} = C \cdot e^{j \omega} w_{\Gamma_1}$$
we define $I_{\pi_2, \Gamma_1}$, as in ii), so that $I_{\pi_2, \Gamma_1} e^{j \omega} w_{\Gamma_1} = C' \cdot e^{j \omega} v_{\Lambda_1}$ 
\vspace{2mm}

\noindent For $x(\pi)v_{\lam}$ we set, as in b), $I_{\pi,\lam} = I_{\pi_2, \Gamma_1}$. 

\item[d)] $x(\pi) =  x_{\tau}(-1)x_{\delta}(-1)$. This monomial nontrivially acts only on $v_{\Lambda_0}$ and $IC$ say that in this case necessarily $\tau = \gamma_{\underline{2}}$ and $\delta = \gamma_2$. We set $I_{\pi, \Lambda_0}$ to be an identity operator.
\end{itemize}

The above defined operators have one important property, namely:
\begin{proposition}
For the operators defined in a) - e) the following holds: if $x(\pi') > x(\pi)$ then $I_{\pi, \Lambda}x(\pi')v_{\Lambda} =0$.
\end{proposition}

\begin{proof}
Let us check this property for the operators defined in a), b) and d) as proof of c) is similar to the proof of b).
\begin{itemize}
 \item [a)] Let $x(\pi)=x_{\gamma_{i_j}}(-j)\dots x_{\gamma_{i_1}}(-1)$, $\gamma_{i_s} \in \{ \gamma_2, \dots, \gamma_{\ell-1} \}$ for $s=1, \dots, j$. We have $v_{\Lambda} = v_{\Lambda_0}$. We will all three operators, $I^1_{\pi, \Lambda_0}$, $I^2_{\pi, \Lambda_0}$ and $I^3_{\pi, \Lambda_0}$ denote as $I_{\pi, \Lambda_0}$. Let $x(\pi') > x(\pi)$. If $x(\pi')$ is greater because its shape is greater, then $x(\pi')$ contains at least two factors of the same degree, i.e. it contains a block of the shape $(-s)(-s)(-(s-1))$. Applying repeatedly Lemma \ref{lem_op_1} we get 
\begin{align*}
& I_{\pi,\Lambda_0} x(\pi') v_{\Lambda_0} = x(\pi')I_{\pi, \Lambda_0}v_{\Lambda_0} = x(\pi') \cdot C \cdot w_{\Gamma_1}= \\ & = \dots x_{\tau_{s+1}}(-s) x_{\tau_s}(-s) \cdot C' \cdot e^{(s-1) \omega} \cdot w_{\Gamma_1''} = 0.
\end{align*}
Let now $x(\pi') = x_{\tau_j}(-j) \dots x_{\tau_1}(-1)$ be of the same shape as $x(\pi)$ and let $x_{\tau_p}(-p)$, $1 \leq p \leq j$, be the first factor from right to left of color greater than $\gamma_{i_p}$. If $\tau_p \notin \{ \gamma_{i_j}, \dots, \gamma_{i_1} \}$ then $\tau_p = \gamma_s$ for some $s \notin \{i_1, \dots, i_j \}$ and this $s \in \Gamma_1$ (see \eqref{operator11}, \eqref{operator12} and \eqref{operator13}). So, by Lemma \ref{lem_op_1} ii) we have 
\begin{align*}
& x(\pi')I_{\pi, \Lambda_0} v_{\Lambda_0} = x(\pi')\cdot C \cdot w_{\Gamma_1}  = \dots x_{\tau_p}(-p) \cdot C' \cdot e^{(p-1) \omega} w_{\Gamma''} = \\ & = \dots x_{\gamma_s}(-p) \cdot C' \cdot e^{(p-1) \omega} w_{\Gamma''} = 0. 
\end{align*}
Namely, we obtained $\Gamma''$ by adding indices to $\Gamma_1$ (see Remark \ref{uzastopna}) and therefore $\Gamma''$ contains index $s$. \\
Now, let $\tau_p = \gamma_{i_r}$. Then necessarily $ r < p$ because $\tau_p > \gamma_{i_p}$. This means the same color appears twice in this first part of $x(\pi')$, i.e. $x(\tau_p)(-p) \dots x_{\tau_1}(-1)$. Lemma \ref{lem_op_1} ii) gives again $x(\pi')I_{\pi, \Lambda_0} v_{\Lambda_0} =0.$
\item[b)] If we observe $x(\pi)$ and $x(\pi')$ on $v_{\Lambda_0}$ then $I_{\pi, \Lambda_0} = I_{\pi2, \Gamma_1} \circ I_{\pi1, \Lambda_0}$ where $\Gamma_1 = \{1,2, \dots, \ell  \}$, i.e. $w_{\Gamma_1} = v_{\llam}$. Since operator $I_{\pi1,\Lambda_0}$ maps $v_{\Lambda_0}$ to $v_{\llam}$, it is enough to show our claim for vector $x(\pi)v_{\llam}$. Now, if $x(\pi')$ is of greater shape than $x(\pi)$ we again have $x(\pi')v_{\llam} =0$. If, on the other hand, $x(\pi)$ and $x(\pi')$ are of the same shape, they must differ in color. Applying Lemma \ref{lem_op_2} we have $x(\pi)v_{\llam} = C \cdot e^{j \omega} w_{\Gamma_1}$, $x(\pi')v_{\llam} = C' \cdot e^{j \omega} w_{\Gamma_1'}$ and $\Gamma_1 \neq \Gamma_1'$. Lemma \ref{svojstvo_op} then gives the required property.
\item[d)] Since $x(\pi) = x_{\gamma_{\underline{2}}}(-1)x_{\gamma_2}(-1)$ there is no greater monomial $x(\pi') > x(\pi)$ such that $x(\pi')v_{\Lambda_0} \neq 0$.
\end{itemize}
\end{proof}

Let $x(\pi)$ be a monomial from $U(\gtl_1)$. With $x(\pi)^{+n}$ we denote a monomial obtained from $x(\pi)$ by lifting the degree in all the factors by $n$. 

Now we can prove the decisive proposition for the proof of the linear independence:
\begin{proposition}\label{operatori}
Let $L(\Lambda)$ be a level 1 module, i.e. $\Lambda \in \{\Lambda_0, \Lambda_1, \Lambda_{\ell-1}, \Lambda_{\ell} \}$, and let $x(\pi)$ be a monomial that satisfies both $DC$ and $IC$ on $W(\Lambda)$ . 

If, on the other hand, $x(\pi) = \dots x_{\delta}(-1)$, then there exist a factorization of monomial $x(\pi)$, $x(\pi)=x(\pi_2)x(\pi_1)$ and an intertwining operator $I_{\pi_1, \Lambda}$ which commutes with the action of $\gtl_1$ so that the following holds:
\begin{itemize}
 \item[a)] $I_{\pi_1, \Lambda}x(\pi_1)v_{\Lambda} = C \cdot e^{n \omega} v_{\Lambda'} \neq 0$ for some $n$ and some $\Lambda'$,
 \item[b)] $I_{\pi_1, \Lambda}x(\pi_1')v_{\Lambda} = 0$ for $x(\pi_1') > x(\pi_1)$,
 \item[c)] $I_{\pi_1, \Lambda}x(\pi_1)v_{\Lambda}= C \cdot e^{n \omega} v_{\Lambda'}$ and $x(\pi_2)^{+n}$ satisfies $DC$ and $IC$ on $W(\Lambda')$.
\end{itemize}

If $x(\pi) = \dots x_{\delta_p}(-p)$ where $p \geq 2$, then there exists an intertwining operator $I_{\pi_0, \Lambda}$ such that:
\begin{itemize}
 \item[a)] $I_{\pi_0, \Lambda}x(\pi)v_{\Lambda} = x(\pi) I_{\pi_0, \Lambda} v_{\Lambda} = x(\pi) \cdot C \cdot e^{(p-1)\omega} v_{\Lambda_0}$,
 \item[b)] $I_{\pi_0, \Lambda}x(\pi_1')v_{\Lambda} = 0$ for $x(\pi_1') = \dots x_{\tau_m}(-m)$, where $m < p$.
\end{itemize}
We will call this operator \textit{a basic skip operator}.
\end{proposition}

\begin{proof}
Suppose first $x(\pi) = \dots x_{\delta}(-1)$. Then the existence of such partition and the affiliated operators with properties a) and b) follows directly from proposition above. We simply set $x(\pi_1)$ to be the first successive blocks in $x(\pi)$ and take the corresponding operator. Property c) is also clear from the way we defined the operators.

Let now $x(\pi) = \dots x_{\delta_p}(-p)$, $p \geq 2$. If $v_{\Lambda}=v_{\Lambda_0}$, we simply set $I_{\pi_0, \Lambda_0} = w(0,p-1,0)$ (see \eqref{op4}). If $v_{\Lambda} = v_{\llam}$ we compose the following basic operators:
\begin{align}
\notag v_{\llam} \overset{w(0,\llam')}{\mapsto} C_1 \cdot v_{\Lambda_1} = C_1' \cdot e^{\omega} v_{\Lambda_0} \overset{w(1,p-1,0)}{\mapsto} C_{(p-1)} \cdot e^{(p-1)\omega}v_{\Lambda_0},
\end{align}
i.e. we set $I_{\pi_0, \llam} = w(1,p-1,0) \circ w(0,\llam)$. Similarly, for $v_{\Lambda} = v_{\lam}$ we have $I_{\pi_0, \lam} = w(1,p-1,0) \circ w(0,\lam)$. If $v_{\Lambda} = v_{\Lambda_1}$, we simply set $I_{\pi_0, \Lambda_1} = w(1,p-1,0)$ so we have:
$$I_{\pi_0, \Lambda_1} v_{\Lambda_1} = w(1,p-1,0) \cdot C \cdot e^{\omega} v_{\Lambda_0} = C' \cdot e^{(p-1) \omega} v_{\Lambda_0}.$$
Property b) follows immediately from Lemma \ref{lem_op_1} since for $m < p$ we have:
$$x(\pi') \cdot C \cdot e^{(p-1)\omega} v_{\Lambda_0} = \dots x_{\tau_m}(-m) \cdot C \cdot e^{(p-1)\omega} v_{\Lambda_0} = C' \cdot x(\pi')^{+(p-1)} v_{\Lambda_0} = 0.$$
\end{proof}

\begin{remark} \label{kompon}
Let $x(\pi)$ be a monomial that satisfies both $DC$ and $IC$ on $W(\Lambda)$ and let $x(\pi') > x(\pi)$. We can then find an intertwining operator $I_{\pi, \Lambda}$ such that
\begin{itemize}
 \item[a)] $I_{\pi, \Lambda}x(\pi)v_{\Lambda} = C \cdot e^{n \omega} v_{\Lambda'} \neq 0$ for some $n$ and some $\Lambda'$,
 \item[b)] $I_{\pi, \Lambda}x(\pi')v_{\Lambda} = 0$.
 \end{itemize}
This follows directly from the Proposition \ref{operatori}, we simply compose the basic operators for the successive blocks in $x(\pi)$ and basic skip operators.
\begin{example} 
Let $\ell = 6$, $\Lambda = \Lambda_0$ and $$x(\pi) = x_{\gamma_{\underline{2}}}(-7) x_{\gamma_{\underline{4}}}(-6) x_{\gamma_{\underline{3}}}(-3) x_{\gamma_5}(-2)x_{\gamma_{3}}(-1)$$
a monomial that satisfies the $DC$ and $IC$ for $W(\Lambda_0)$. Then for the first successive block, i.e. 
$x_{\gamma_5}(-2)x_{\gamma_{3}}(-1)$ we have a basic intertwining operator $I_{\pi_1, \Lambda_0}$ such that:
\begin{align} \label{pr1}
& I_{\pi_1, \Lambda_0} x(\pi) v_{\Lambda_0} = C_1 \cdot x(\pi) w_{1246} = \\ & = \notag  C_1 \cdot  x_{\gamma_{\underline{2}}}(-7) x_{\gamma_{\underline{4}}}(-6) x_{\gamma_{\underline{3}}}(-3)  x_{\gamma_5}(-2)x_{\gamma_{3}}(-1) w_{1246} = \\ & = \notag x_{\gamma_{\underline{2}}}(-7) x_{\gamma_{\underline{4}}}(-6) x_{\gamma_{\underline{3}}}(-3) \cdot C_2 \cdot  e^{2 \omega} w_{123456} \\ & \notag  =  
x_{\gamma_{\underline{2}}}(-7) x_{\gamma_{\underline{4}}}(-6) \cdot C_3 \cdot e^{2 \omega}x_{\gamma_{\underline{3}}}(-1) w_{123456} = \\ & \notag =   
x_{\gamma_{\underline{2}}}(-7) x_{\gamma_{\underline{4}}}(-6) \cdot C_4 \cdot e^{ 3 \omega} w_{12456} 
\end{align}
\end{example}
Now we use the basic operator $I_{\pi_1, \Gamma_1}$ for $\Gamma_1 = \{12456 \}$ which maps $e^{3 \omega}w_{12456}$ to $e^{3 \omega}v_{\Lambda_1}$ and get:
\begin{align} \label{pr2}
& I_{\pi_1, \Gamma_1} \hspace{0.5mm} x_{\gamma_{\underline{2}}}(-7) x_{\gamma_{\underline{4}}}(-6) \cdot C_4 \cdot e^{ 3 \omega} w_{12456} = x_{\gamma_{\underline{2}}}(-7) x_{\gamma_{\underline{4}}}(-6) \cdot C_5 \cdot e^{ 3 \omega} v_{\Lambda_1} = \\ \notag & = C_6 \cdot x_{\gamma_{\underline{2}}}(-7) x_{\gamma_{\underline{4}}}(-6) e^{ 4 \omega} v_{\Lambda_0}.
\end{align}
We apply the basic skip operator $I_{\pi_0, \Lambda_0}$ such that  $I_{\pi_0, \Lambda_0}e^{4 \omega} v_{\Lambda_0} = C \cdot e^{ 5 \omega} v_{\Lambda_0}$ and we have:
\begin{align} \label{pr3}
C_6 \cdot I_{\pi_0, \Lambda_0} \hspace{0.5mm}  x_{\gamma_{\underline{2}}}(-7) x_{\gamma_{\underline{4}}}(-6) e^{ 4 \omega} v_{\Lambda_0} = C_7 \cdot x_{\gamma_{\underline{2}}}(-7) x_{\gamma_{\underline{4}}}(-6) e^{ 5 \omega} v_{\Lambda_0}.
\end{align}
Now we send $e^{5 \omega} v_{\Lambda_0}$ to $e^{5 \omega} w_{123456} = e^{5 \omega} v_{\llam}$ with $I_{\pi_2, \Lambda_0}$ and we get:
\begin{align} \label{pr4}
& I_{\pi_2, \Lambda_0} \hspace{0.5mm} C_7 \cdot x_{\gamma_{\underline{2}}}(-7) x_{\gamma_{\underline{4}}}(-6) e^{ 5 \omega} v_{\Lambda_0} = x_{\gamma_{\underline{2}}}(-7) x_{\gamma_{\underline{4}}}(-6) \cdot C_8 \cdot e^{ 5 \omega} v_{\llam} = \\ \notag & = C_9 \cdot e^{5 \omega} x_{\gamma_{\underline{2}}}(-2) x_{\gamma_{\underline{4}}}(-1)  w_{123456} =  C_{10} \cdot e^{7 \omega} w_{1356} 
\end{align}
We apply one more operator, $I_{\pi_2, \Gamma_1}$ where $\Gamma_1 = \{1356 \}$ and we have
\begin{align} \label{pr5}
& I_{\pi_2, \Gamma_1} \hspace{0.5mm}  C_{10} \cdot e^{7 \omega} w_{1356} = C_{11} \cdot e^{p \omega} v_{\Lambda_1}. 
\end{align}
The number $p$ depends on the shape of monomial $x(\pi')$. For instance, if $x(\pi') = \dots x_{\tau}(-10) x(\pi)$, we take $p \geq 9$ because then $\dots x_{\tau}(-9) e^{p \omega} v_{\Lambda_1} =0.$

The final operator $I_{\pi, \Lambda_0}$ is now the composition of operators \eqref{pr1}, \eqref{pr2}, \eqref{pr3}, \eqref{pr4} and 
\eqref{pr5}.
\end{remark}

\section{Proof of linear independence in level 1 case}\label{kraj1}
We will now use Proposition \ref{operatori} to prove the linear independence of the admissible vectors for $W(\Lambda)$ (see \eqref{razap1}). In this way we get the basis of the Feigin-Stoyanovsky subspace.
\begin{theorem}
Let $L(\Lambda_i)$, $i=0,1, \ell-1, \ell$ be a standard $\gtl$-module of level 1. The set of admissible vectors
$$\{ \ x(\pi)v_{\Lambda_i} \ \vert \ x(\pi) \ \textrm{satisfies DC and IC for $W(\Lambda_i)$} \ \}$$
is a basis of $W(\Lambda_i)$.
\end{theorem}

\begin{proof}
By Proposition \ref{razap} we know that this set spans $W(\Lambda_i)$.
We shall prove the linear independence by induction on degree and order of monomials simultaneously for all level 1 standard modules. Assume
\begin{align}\label{suma}
 \sum c_{\pi'} x(\pi') v_{\Lambda_i} = 0
\end{align}
where all the monomials $x(\pi')$ satisfy $DC$ and $IC$ for $W(\Lambda_i)$. Assume further on all monomials are of degree greater than $-n$ and that $c_{\pi'} = 0$ for $x(\pi') < x(\pi)$ where $x(\pi)$ is some fixed monomial. We want to show that $c_{\pi}=0$. 
\vspace{2mm}

Suppose first $x(\pi) = \dots x_{\delta}(-m)$ where $m \geq 2$. If $v_{\Lambda_i} = v_{\Lambda_0}$, we use the operator from the third group which maps $v_{\Lambda_0}$ to $C \cdot v_{\Lambda_1}$, see \eqref{op3}. If $v_{\Lambda_i} = v_{\Lambda_{\ell-1}}$ or $v_{\Lambda_i} = v_{\Lambda_{\ell}}$, we use the operators from the second group which map $v_{\Lambda_{\ell-1}}$ to $C \cdot v_{\Lambda_1}$, that is $v_{\Lambda_{\ell}}$ to $C \cdot v_{\Lambda_1}$, see \eqref{op2}. In all this cases we get, after applying the operator:
\begin{align} 
 \sum c_{\pi'} x(\pi') v_{\Lambda_1} = 0.
\end{align}
(if $v_{\Lambda_i}$ was equal to $v_{\Lambda_1}$, we needn't have applied any operator, of course).  Now we use the fact that $v_{\Lambda_1} = C \cdot e^{\omega}v_{\Lambda_0} = C' \cdot e(\omega)v_{\Lambda_0}$. Operator $e(\omega)$ commutes to the left with monomials and lifts the degree of each factor by one, i.e. we have
\begin{align} 
 \sum c_{\pi'} x(\pi') v_{\Lambda_1} = C' \cdot \sum c_{\pi'} x(\pi') e(\omega) v_{\Lambda_0} = C' \cdot e(\omega) \sum c_{\pi'} x(\pi')^{+1}  v_{\Lambda_0} = 0.
\end{align}
Operator $e(\omega)$ is an injection and therefore we get:
\begin{align} \label{suma1}
 \sum c_{\pi'} x(\pi')^{+1}  v_{\Lambda_0} = 0.
\end{align}
All the monomials in \eqref{suma1} satisfy the $DC$ since the degree of each factor in \eqref{suma} was lifted by one. The also satisfy the $IC$ for $W(\Lambda_0)$ because $x_{\underline{\delta}}(-2)x_{\delta}(-2)$ block doesn't satisfy the $DC$ if $\delta \neq \gamma_2$. Using the induction presumption we now conclude $c_{\pi} = 0$. 
\vspace{2mm}

Suppose now $x(\pi) = \dots x_{\delta}(-1)$. By Proposition \ref{operatori} there exist a factorization of monomial $x(\pi)$, $x(\pi)=x(\pi_2)x(\pi_1)$ an operator $I_{\pi_1, \Lambda}$ such that:
\begin{itemize}
 \item[-] $I_{\pi_1, \Lambda_i}x(\pi_1)v_{\Lambda_i} = C \cdot e^{n \omega} v_{\Lambda_j} \neq 0$ where $v_{\Lambda_j}$ is also a highest weight vector of a level 1 module,
 \item[-] $I_{\pi_1, \Lambda_i}x(\pi_1')v_{\Lambda_i} = 0$ for $x(\pi_1') > x(\pi_1)$ and
 \item[-] $I_{\pi_1, \Lambda_i}x(\pi_1)v_{\Lambda_i}= C \cdot e^{n \omega} v_{\Lambda_j}$ and $x(\pi_2)^{+n}$ satisfies $DC$ and $IC$ on $W(\Lambda_j)$.
\end{itemize}
Applying this operator to the equation \eqref{suma} gives:
\begin{align*} 
0 &= I_{\pi_1, \Lambda_i} \left( \sum c_{\pi'} x(\pi') v_{\Lambda_i} \right) = I_{\pi_1, \Lambda_i} \left(\sum_{x(\pi'_1)< x(\pi_1)} c_{\pi'}x(\pi')v_{\Lambda_i} \right) + \\ & + I_{\pi_1, \Lambda_i} \left(\sum_{x(\pi'_1)> x(\pi_1)} c_{\pi'}x(\pi')v_{\Lambda_i} \right) + I_{\pi_1, \Lambda_i} \left(\sum_{x(\pi_1')= x(\pi_1)} c_{\pi'}x(\pi')v_{\Lambda_i} \right).
\end{align*}
First sum is equal to zero by presumption. The second becomes zero after $I_{\pi_1, \Lambda_i}$ acts on $x(\pi_1')v_{\Lambda_i}$. What is left is:
\begin{align*}
0 &= I_{\pi_1, \Lambda_i} \left(\sum_{x(\pi_1')= x(\pi_1)} c_{\pi'}x(\pi')v_{\Lambda_i} \right) =
 \sum_{x(\pi_1')= x(\pi_1)} c_{\pi'}x(\pi_2')I_{\pi} x(\pi_1)v_{\Lambda_i}  =\\ &= 
\sum_{x(\pi_1')= x(\pi_1)} c_{\pi'}x(\pi_2') \cdot C \cdot e(p \omega)v_{\Lambda_j} = C \cdot e(p \omega) \sum_{x(\pi_1')= x(\pi_1)} c_{\pi'}x(\pi_2')^{+p}v_{\Lambda_j}
\end{align*}
Since $e(p \omega)$ is an injection it follows that:
\begin{align}
 \sum_{x(\pi_1')= x(\pi_1)} c_{\pi'}x(\pi_2')^{+p}v_{\Lambda_j} = 0.
\end{align}
All monomials in this sum satisfy the difference conditions and are of degree greater or equal $-n+1$. They do not necessarily satisfy the initial conditions on $v_{\Lambda_j}$ but in this case they annihilate the highest weight vector. Namely, if a monomial $x(\mu)$ doesn't satisfy the $IC$ for $W(\Lambda)$ for some $\Lambda$ and at the same time $x(\mu)v_{\Lambda} \neq 0$ then it has to be $\Lambda = \Lambda_0$ and $ x(\mu) = \dots x_{\underline{\delta}}(-1)x_{\delta}(-1)$, $\delta \neq \gamma_2$. But this can not happen as  $x_{\underline{\delta}}(-(p+1))x_{\delta}(-(p+1))$ is not allowed by $DC$ if $\delta \neq \gamma_2$. Our fixed monomial $x(\pi)^{+p}$ satisfies the $IC$ according to the Proposition \ref{operatori} so we conclude $c_{\pi}=0$.
\end{proof}
Now we will use the same approach to find base of Feigin-Stoyanovsky's type subspaces of the level 2 modules for algebra $D_4^{(1)}$. We will again define the difference and initial conditions and then, using the intertwining operators, prove the linear independence of the spanning set. First we need the realization of the level 2 modules as submodules of tensor products of level 1 modules.
\section{Level 2 modules as subspaces of tensor products of modules $L(\lam)$ and $L(\llam)$}\label{pocetak2}
Standard modules $L(\Lambda_2), \dots, L(\Lambda_{\ell-2})$ for affine Lie algebra of type $D_{\ell}$ are of level 2. We realize this modules as submodules of the tensor product of level 1 modules. We need the following lemma:
\begin{lemma}\label{dekomp}
Let $\omega_1, \omega_2, \dots, \omega_{\ell-1}, \omega_{\ell}$ be the fundamental weights of a finite dimensional Lie algebra $\g$ of type $D_{\ell}$. The tensor products of spinor representations then decompose as follows:
\begin{align}
\label{ten1} L(\omega_{\ell-1}) \otimes L(\omega_{\ell}) &= L(\omega_{\ell-1} + \omega_{\ell}) \oplus L(\omega_{\ell-3}) \oplus L(\omega_{\ell-5}) \oplus \dots, \\
\label{ten2} L(\omega_{\ell-1}) \otimes L(\omega_{\ell-1}) &= L(2\omega_{\ell-1})  \oplus L(\omega_{\ell-2}) \oplus L(\omega_{\ell-4}) \oplus \dots, \\
\label{ten3} L(\omega_{\ell}) \otimes L(\omega_{\ell}) &= L(2\omega_{\ell})  \oplus L(\omega_{\ell-2}) \oplus L(\omega_{\ell-4}) \oplus \dots. 
\end{align}
\end{lemma}

\begin{proof}
We use the formula for the decomposition of the tensor product of two irreducible modules, $V(\lambda') \otimes V(\lambda'') $, in case when we know the weights and multiplicities of at least one of them (see \cite{H}),
$$ch_{\lambda'} * ch_{\lambda''} = \sum_{\lambda \in \Pi(\lambda')} m_{\lambda'}(\lambda)t(\lambda + \lambda'' + \delta)ch_{ \{ \lambda + \lambda'' + \delta \} - \delta}$$
where
$\Pi(\lambda')$ denotes the set of the weights of the module $V(\lambda')$ and $t(\mu)$ is defined as follows:
\begin{itemize}
 \item [-] $t(\mu) = 0$ if some element of Weyl's group $W$ different from identity fixes weight $\mu$,
 \item [-] $t(\mu) = sn(\sigma)$ if only identity fixes $\mu$ and $\sigma \in W$ is such that $\sigma(\mu)$ is a dominant weight.
\end{itemize}
Further on, $\{ \mu \}$ denotes the unique dominant weight conjugate of $\mu$. To prove \eqref{ten1}, we set $V(\lambda') = L(\omega_{\ell-1})$ and $V(\lambda'') = L(\omega_{\ell})$. Other decompositions are proved in a similar way.
\end{proof}
If we now take a look at the $\gtl$-modules $L(\lam)$ and $L(\llam)$, we see that the finite dimensional representations $L(\omega_{\ell-1})$ and $L(\omega_{\ell})$ are on top of those modules with respect to the gradation \eqref{gradacija}. Therefore we have on top of the $L(\lam) \otimes L(\llam)$,  $L(\lam) \otimes L(\lam)$ and   $L(\llam) \otimes L(\llam)$ the irreducible representations from decompositions \eqref{ten1}, \eqref{ten2} and \eqref{ten3} respectively. Now, the highest weight vectors for those $\g$-modules will be the highest weight vectors of the $\gtl$-modules $L(\Lambda_j)$, $j=2, \dots, \ell-2$, as well. If we set $m = \ell-j$, $m \geq 3$, we can therefore conclude that:
\begin{align}
L(\Lambda_j) &< L(\lam) \otimes L(\lam) \ \textrm{i.e.} \ L(\Lambda_j) < L(\llam) \otimes L(\llam) \quad \textrm{if $m$ is even}, \\
L(\Lambda_j) &< L(\lam) \otimes L(\llam) \quad \textrm{if $m$ is odd}.
\end{align}
How do those highest weight vectors look like? Weight vectors of the spinor representations are of the form $w_{\Psi}$ where $\Psi \subset \{1,2, \dots, \ell \}$ (see Section \ref{spinor}). Therefore vectors $w_{\Psi_1} \otimes w_{\Psi_2}$ will be of the weight $\omega_j = \epsilon_1 + \dots + \epsilon_j$, $j \in \{2, \dots, \ell-2 \}$, if 
\begin{itemize}
 \item [1)] $\{ 1, \dots, j \} \subset \Psi_1$ and $\{ 1, \dots, j \} \subset \Psi_2$,
 \item [2)] $\Psi_1 - \{ 1, \dots, j \}$ and $\Psi_2 - \{ 1, \dots, j \}$ are disjoint sets such that 
 \\ $(\Psi_1 - \{ 1, \dots, j \}) \cup (\Psi_2 - \{ 1, \dots, j \}) = \{j+1, \dots, \ell \}$.
\end{itemize}
\begin{example}
Let $\ell=7$. Then vector $w_{123456} \otimes w_{12347}$ is of weight $\omega_4$ in $L(\Lambda_6) \otimes L(\Lambda_7)$. Vector $w_{12346} \otimes w_{12357}$ is of weight $\omega_3$ in $L(\Lambda_7) \otimes L(\Lambda_7)$ and vector $w_{1236} \otimes w_{123457}$ is of weight $\omega_3$ in $L(\Lambda_6) \otimes L(\Lambda_6)$.
\end{example}
We will call such vectors \textit{basic vectors} and their linear combination will give us the highest weight vectors. The following lemma is proved in a similar way as Lemma \ref{lem_op_1}, using formula \eqref{formula}.

\begin{lemma}
Let $\alpha=\epsilon_i - \epsilon_j$. Then $x_{\alpha}(0)w_{\Gamma_1} \neq 0$ if and only if $i \notin \Gamma_1, j \in \Gamma_1$ in which case $x_{\alpha}(0)w_{\Gamma_1} = C \cdot w_{\Gamma_1'}$ where $\Gamma_1' = \Gamma_1 \cup \{ i\} - \{ j\}$. If $\alpha=\epsilon_i + \epsilon_j$, then $x_{\alpha}(0)w_{\Gamma_1} \neq 0$ if and only if $i \notin \Gamma_1, j \notin \Gamma_1$ in which case $x_{\alpha}(0)w_{\Gamma_1} = C \cdot w_{\Gamma_1'}$ where $\Gamma_1' = \Gamma_1 \cup \{ i,j\}$. If $\alpha=-\epsilon_i - \epsilon_j$, then $x_{\alpha}(0)w_{\Gamma_1} \neq 0$ if and only if $i,j \in \Gamma_1$ in which case $x_{\alpha}(0)w_{\Gamma_1} = C \cdot w_{\Gamma_1'}$ where $\Gamma_1' = \Gamma_1 - \{ i,j\}$.
\end{lemma}

Using this lemma we can prove:
\begin{proposition}\label{vektor_j}
Let $j= \ell - m$, $m \geq 3$ odd. Then $L(\lam) \otimes L(\llam)$ contains a submodule $L(\Lambda_j)$ with a highest weight vector of the form
\begin{align}
\label{vektor} v_{\Lambda_j} = \sum_p C_p \cdot w_{12 \dots j \Psi_1^p} \otimes w_{12 \dots j \Psi_2^p}
\end{align}
where sets $\Psi_1^p$ and $\Psi_2^p$ are disjoint subsets of $\{ j+1, \dots, \ell \}$ such that $\Psi^p_1 \cup \Psi^p_2 = \{ j+1, \dots, \ell \}$. Further on, $\vert \Psi_1^p \vert + j \equiv \ell - 1 \ \textrm{mod} \ 2$ and $\vert \Psi_2^p \vert + j \equiv \ell \ \textrm{mod} \ 2$. All coefficients in \eqref{vektor} are nonzero. The analogous claims hold for the submodules $L(\Lambda_j) < L(\lam) \otimes L(\lam)$ (i.e. $L(\Lambda_j) < L(\llam) \otimes L(\llam)$).
\end{proposition}

\begin{proof}
We need to show that coefficients in \eqref{vektor} differ from zero. Let

\begin{align}
\label{vektor1} v_{\Lambda_j} = \sum_p C_p \cdot w_{12 \dots j \Psi_1^p} \otimes w_{12 \dots j \Psi_2^p}
\end{align}
and let $w_{12 \dots j \Psi_1} \otimes w_{12 \dots j \Psi_2}$ be the vector with nonull coefficient from this linear combination such that its $\Psi_2$ set has the maximal number of elements (this vector needn't be unique, we pick one of them). If $w_{12 \dots j\Psi_2} \neq w_{12 \dots \ell}$, then we can find two elements $i_1$ and $i_2$ such that $i_1, i_2 \in \{j+1, \dots, \ell \} - \Psi_2$. We apply $x_{\alpha}(0)$ to \eqref{vektor1} where $\alpha = \epsilon_{i_1} + \epsilon_{i_2}$ and we get, among others, vector $w_{12 \dots j \Psi_1'} \otimes w_{12 \dots j \Psi_2'}$ where $\Psi_2' = \Psi_2 \cup \{i_1, i_2 \}$. It is clear that this vector could be obtained exclusively from $w_{12 \dots j \Psi_1} \otimes w_{12 \dots j \Psi_2}$ because $\vert \Psi_2 \vert$ is maximal. Therefore we get $x_{\alpha}(0)v_{\Lambda_j} \neq 0$ which is a contradiction. We conclude that $\{ 12 \dots j \Psi_2 \} = \{12, \dots, \ell \}$, i.e. vector $w_{12 \dots j} \otimes w_{12 \dots \ell}$ appears with nonull coefficient in \eqref{vektor1}. 

Suppose now there exist a vector $w = w_{12 \dots j \Psi_1} \otimes w_{12 \dots j \Psi_2}$ such that $\Psi_1$ has two elements, i.e. $\Psi_1 = \{ t_1, t_2 \}$ and such that the corresponding coefficient for $w$ in \eqref{vektor1} equals zero. We apply $x_{\alpha}(0)$ to \eqref{vektor1} where $\alpha = \epsilon_{t_1} + \epsilon_{t_2}$. This transfers $w_{12 \dots j} \otimes w_{12 \dots \ell}$ to $w_{12 \dots j t_1 t_2} \otimes w_{12 \dots \ell}$ and again this is the only way to obtain vector $w_{12 \dots j t_1 t_2} \otimes w_{12 \dots \ell}$ since the sum \eqref{vektor1}  doesn't contain $w$. It follows that $x_{\alpha}(0)v_{\Lambda_j} \neq 0$ which is a contradiction. Therefore all vectors $w_{12 \dots j \Psi_1} \otimes w_{12 \dots j \Psi_2}$ such that $\Psi_1$ has two elements appear with a nonull coefficient in \eqref{vektor1}. We proceed in an analogous way, proving that all vectors of the form $w_{12 \dots j \Psi_1} \otimes w_{12 \dots j \Psi_2}$ where $\Psi_1$ has four elements appear with nonull coefficient in \eqref{vektor1} an so on.
\end{proof}
Let us state one more fact concerning the vectors $v_{\Lambda_j}$:
\begin{proposition}\label{kolinearnost}
Denote by $L(\Lambda_j)$ the submodule of $L(\lam) \otimes L(\llam)$ with the highest weight vector of the form \eqref{vektor} and let $V$ be its invariant complement. Then $L(\Lambda_j)\cong L(\lam) \otimes L(\llam) / V$ and in $L(\Lambda_j)$ we have $w_{12 \dots j \Psi_1^p} \otimes w_{12 \dots j \Psi_2^p} = C_{ps} \cdot w_{12 \dots j \Psi_1^s} \otimes w_{12 \dots j \Psi_2^s}$ for some $C_{ps} \neq 0$, i.e. vectors from \eqref{vektor} are collinear on the quotient $L(\lam) \otimes L(\llam) / V$. The analogous claims hold for the submodules $L(\Lambda_j) < L(\lam) \otimes L(\lam)$ (i.e. $L(\Lambda_j) < L(\llam) \otimes L(\llam)$).
\end{proposition}
\begin{proof}
Let suppose first $L(\Lambda_j) < L(\lam) \otimes L(\llam)$. Vector $v=w_{12 \dots j  i_1  i_2} \otimes w_{12 \dots \ell}$ is of weight $\epsilon_1 + \dots \epsilon_j + \epsilon_{i_1} + \epsilon_{i_2}$ and as such, has to be zero in $L(\Lambda_j)$. Therefore we have $0=x_{\alpha}(0)v$ in $L(\Lambda_j)$ for $\alpha = -(\epsilon_{i_1} + \epsilon_{i_2})$ which gives us 
$$w_{12 \dots j } \otimes w_{12 \dots \ell} - C \cdot w_{12 \dots j  i_1  i_2} \otimes w_{12 \dots j \Psi_2} = 0$$
where $\Psi_2 = \{j+1, \dots, \ell \} - \{i_1,i_2 \}$. So, all basic vectors of the form $w_{12 \dots j  i_1  i_2} \otimes w_{12 \dots j \Psi_2} = 0$ are collinear with $w_{12 \dots j } \otimes w_{12 \dots \ell}$. Now we prove in the similar way that $w_{12 \dots j \ i_1 \ i_2 \ i_3 \ i_4} \otimes w_{12 \dots j \Psi_2'}$ where $\Psi_2 = \{j+1, \dots, \ell \} - \{i_1, i_2, i_3, i_4 \}$ are collinear with $w_{12 \dots j \ i_1 \ i_2} \otimes w_{12 \dots j \Psi_2}$ and so on.
Proof is the same for $L(\Lambda_j) < L(\lam) \otimes L(\lam)$ (i.e. $L(\Lambda_j) < L(\llam) \otimes L(\llam)$).
\end{proof}
In the next section we will use the above description of the highest weight vectors for the level 2 modules to find the difference and initial condition which will enable us to reduce the spanning set \eqref{PBW}.

\begin{remark} \label{e_2}
In the reminder of the text we will denote the vector $e^{\lambda}w \otimes e^{\lambda}w'$ simply with $e^{\lambda}(w \otimes w')$. Notation $e(\lambda)(w \otimes w')$ will be applied for $e(\lambda) \otimes e(\lambda) (w \otimes w') = e(\lambda) w \otimes e(\lambda) w'$ (see Section \ref{op_pr_struje}). Vectors $e^{\lambda}(w \otimes w')$ and $e(\lambda)(w \otimes w')$ are now again collinear. 
\end{remark}

\section{Difference and initial conditions for level 2 modules}
The aim of this section is to reduce the spanning set \eqref{PBW} as in level 1 case (see Section \ref{uvjeti_razlike}). We find the relations among vertex operators which give us the leading terms. Monomials which don't contain the leading terms are said to satisfy the difference conditions. Then we define the initial conditions and show that every monomial which doesn't satisfy the initial conditions can be expressed as a sum of the greater ones.

\subsection{Difference conditions}
In order to find the relations among fields we first observe that:
\begin{lemma} For $\gamma, \delta, \tau, \mu \in \Gamma$ we have:
\begin{align}
x_{\tau}(-1)x_{\delta}(-1)x_{\gamma}(-1) & (v_{\Lambda_0}  \otimes v_{\Lambda_0})  = 0 \\ \notag &  \quad \textrm{if $\{\tau, \delta, \gamma\}$ set doesn't contain $\underline{\mu}, \mu$ pair}, \\ x_{\tau}(-1)x_{\underline{\gamma}}(-1)x_{\gamma}(-1) & (v_{\Lambda_0}  \otimes v_{\Lambda_0}) = C  \cdot x_{\tau}(-1)x_{\underline{\delta}}(-1)  x_{\delta}(-1)(v_{\Lambda_0} \otimes v_{\Lambda_0}) \\ \notag & \quad \textrm{for some $C \neq 0$}.
\end{align}

\end{lemma}
\begin{proof}
Immediately from Lemma \ref{vakuum0}.
\end{proof}
This gives us the relations between the vertex operators:
\begin{proposition} For $\gamma, \delta, \tau, \mu \in \Gamma$ we have:
\begin{align}
\label{rel21} x_{\tau}(z)x_{\delta}(z)x_{\gamma}(z) &= 0  \quad \textrm{if $\{\tau, \delta, \gamma \}$ set doesn't contain $\underline{\mu}, \mu$ pair}, \\ \label{rel22} x_{\tau}(z)x_{\underline{\gamma}}(z)x_{\gamma}(z\index{\footnote{}}) &= C \cdot x_{\tau}(z)x_{\underline{\delta}}(z)x_{\delta}(z) \ \textrm{for some $C \neq 0$}.
\end{align}
\end{proposition}
Using this relations, i.e. regarding the coefficients of $z^{n-3}$, we can again extract the minimal monomials and express them as the sum of the greater ones. 

For instance, \eqref{rel21} can be written as:
\begin{align}
0=x_{\tau}(z)x_{\delta}(z)x_{\gamma}(z) = \sum_n \left( \sum_{p+q+s=n} x_{\tau}(-s)x_{\delta}(-q)x_{\gamma}(-p) \right) z^{-n-3}.
\end{align}

If we suppose that $\gamma \geq \delta \geq \tau$ and $n=3j$, the minimal monomial is \\ $x_{\tau}(-j)x_{\delta}(-j)x_{\gamma}(-j)$. If $n=3j+1$, monomial $x_{\gamma}(-j-1)x_{\tau}(-j)x_{\delta}(-j)$ is the minimal one and so on.

We call this minimal monomials \textit{the leading terms} (as before) and we will say that a monomial $x(\pi)$ satisfies \textit{difference conditions (or $DC$ in short) for level 2} if it doesn't contain the leading terms. Analyzing the coefficients of relations \eqref{rel21} and \eqref{rel22}, as in Section \ref{uvjeti_razlike}, we get the following description of the monomials that satisfy the $DC$:
\begin{proposition}\label{uv_raz_2}
If a monomial $x(\pi)$ satisfies the difference conditions for level 2, then it satisfies the difference conditions for level 1 when we regard every second term from right to left except in the cases when $x(\pi)$ contains factors
\begin{itemize}
 \item[1)] $x_{\delta}(-j-1)x_{\underline{\delta}}(-j)x_{\delta}(-j)$,
 \item[2)] $x_{\underline{\delta}}(-j-1)x_{\delta}(-j)x_{\underline{\delta}}(-j)$
\end{itemize}
for $\delta \in \Gamma$. We call these cases \textit{the exceptional DC}.
\end{proposition}

It is easy to see that this conditions can again be written in another form.
\begin{proposition}\label{uvj_razl_dr_zap}
The difference conditions are equivalent to the following frequency conditions:
\begin{itemize} 
\item[1)] $b_r + b_{r-1} + \dots + b_2 + a_{\underline{2}} + \dots
+ a_{\underline{\ell}} +  a_\ell + \dots + a_{r+1}     \leq 2, \quad r =
2, \dots, \ell-1$, \item[2)] $b_r + \dots + b_2 + a_{\underline{2}} +
\dots+a_{\underline{r-1}} + a_{\underline{r+1}} \dots+
a_{\underline{\ell}} +  a_\ell + \dots + a_r \leq 2$,
\quad $r = 2, \dots, \ell-1$,
\item[3)] $b_\ell + b_{\ell-1} + \dots + b_2 + a_{\underline{2}} + \dots
+ a_{\underline{\ell-1}} +  a_\ell \leq 2$, \item[4)] $b_{\underline{\ell}}
+ b_{l-1} + \dots + b_2 + a_{\underline{2}} + \dots +
a_{\underline{\ell-1}} +  a_{\underline{\ell}} \leq 2$, \item[5)]
$b_{\underline{r+1}} + \dots + b_{\underline{\ell}} + b_{\ell} + \dots +
b_2 + a_{\underline{2}} + \dots + a_{\underline{r}} \leq 2$,
\quad $r = \ell-1, \dots, 2$,
\item[6)] $b_{\underline{r}} + b_{\underline{r+1}} \dots +
 b_{r+1} + b_{r-1} + \dots +
b_2 + a_{\underline{2}} + \dots + a_{\underline{r}} \leq 2$,
\quad $r = \ell-1, \dots, 2$.
\end{itemize}
Here $b_r$ (resp. $b_{\underline{r}}$) denotes the number of $x_{\gamma_r}(-j-1)$ (resp. $x_{\gamma_{\underline{r}}}(-j-1)$) factors and $a_r$ (resp. $a_{\underline{r}}$) of $x_{\gamma_r}(-j)$ (resp. $x_{\gamma_{\underline{r}}}$) factors in monomial $x(\pi)$ for any given $j$.
\end{proposition}

From Proposition \ref{uv_raz_2} an important property of the monomials that satisfy $DC$ follows:
\begin{proposition}\label{particija}
Let $x(\pi)$ be a monomial that satisfies $DC$ for level 2. Then 
\begin{itemize}
 \item [a)] either $x(\pi)$ has a partition on two subarrays, $x(\pi_1)$ and $x(\pi_2)$ such that both of those subarrays satisfy difference conditions for level one,
\item [b)] or $x(\pi)$ contains blocks of the form $x_{\delta}(-j-1)x_{\underline{\delta}}(-j)x_{\delta}(-j)$ or $x_{\underline{\delta}}(-j-1)x_{\delta}(-j)x_{\underline{\delta}}(-j)$ where $\delta \in \{ \gamma_3, \dots, \gamma_{\ell-1} \}$, and parts before and after this blocks have partitions as in a).
\end{itemize}
 \end{proposition}
\begin{proof}
From Proposition \ref{uv_raz_2} we have that either $x(\pi)$ satisifes on every second term from right to left the DC for level one or $x(\pi)$ has blocks that satisfy the exceptional $DC$. If there are no such exceptional $DC$ blocks, we can simply take every second term and we get the two subarrays $x(\pi_1)$ and $x(\pi_2)$. 

\begin{example}
Let $\ell = 5$ and $$x(\pi) = x_{\gamma_3}(-4)x_{\gamma_{\underline{4}}}(-2)x_{\gamma_{5}}(-2)x_{\gamma_{\underline{5}}}(-1)x_{\gamma_3}
(-1).$$ 
If we now take every second factor, looking from right to left, we get the two subbarrays $x(\pi_1)$ and $x(\pi_2)$ which satisfy the $DC$ for level 1:
\begin{align*}
x(\pi_1) &= x_{\gamma_3}(-4) x_{\gamma_{5}}(-2) x_{\gamma_3}(-1), \\
x(\pi_2) &= x_{\gamma_{\underline{4}}}(-2)x_{\gamma_{\underline{5}}}(-1).
\end{align*}
\end{example}
If, on the other hand, monomial $x(\pi)$ contains the blocks which satisfy the exceptional $DC$, then either $x(\pi)$ belongs to the case b) or this blocks are of the form
\begin{align} \label{iznimke}
& x_{\gamma_2}(-j-1)x_{\gamma_{\underline{2}}}(-j)x_{\gamma_2}(-j), \\ \notag  & x_{\gamma_{ \ell}}(-j-1)x_{\gamma_{\underline{\ell}}}(-j)x_{\gamma_{\ell}}(-j), \\ \notag & x_{\gamma_{\underline{2}}}(-j)x_{\gamma_2}(-j)x_{\gamma_{\underline{2}}}(-(j-1)) \quad \textrm{or} \\ \notag & x_{\gamma_{\underline{\ell}}}(-j)x_{\gamma_{\ell}}(-j)x_{\gamma_{\underline{\ell}}}(-(j-1)).
\end{align}

Since the monomials which contain the blocks above are not covered by b), we have to prove they satisfy the claim a), i.e. we have to prove the existence of partitions $x(\pi_1)$ and $x(\pi_2)$ that satisfy the $DC$ for level 1 for those monomials. 

Suppose now monomial $x(\pi)$ contains a block of the form \\ $x_{\gamma_{ \ell}}(-j-1)x_{\gamma_{\underline{\ell}}}(-j)x_{\gamma_{\ell}}(-j)$. Suppose further on that $x_{\gamma_{ \ell}}(-j-1)x_{\gamma_{\underline{\ell}}}(-j)x_{\gamma_{\ell}}(-j)$ is first such block from right to left, i.e. that $x(\pi)$ is of the form 
$$x(\pi) = \dots x(\pi'') x_{\gamma_{ \ell}}(-m-1)x_{\gamma_{\underline{\ell}}}(-m)x_{\gamma_{\ell}}(-m) \dots x_{\gamma_{ \ell}}(-j-1)x_{\gamma_{\underline{\ell}}}(-j)x_{\gamma_{\ell}}(-j)x(\pi')$$
where $x(\pi')$ and $x(\pi'')$ satisfy condition from a), i.e. they satisfy $DC$ for level 1 on every second factor. We partition $x_{\gamma_{\underline{\ell}}}(-j)x_{\gamma_{\ell}}(-j)x(\pi')$ in two subarrays that satisfy $DC$ for level 1 (taking every second factor), lets name them $x(\pi'_1) = x_{\gamma_{\ell}}(-j) \dots$ and $x(\pi'_2) = x_{\gamma_{\underline{l}}}(-j) \dots$. Now we have to settle the $x_{\gamma_{\ell}}(-j-1)$ factor. This factor should be added on $x(\pi_1')=x_{\gamma_{\ell}}(-j) \dots $ if we continue to take every second term, but this will obviously violate the difference conditions. So, we put $x_{\gamma_{\ell}}(-j-1)$ on $x(\pi_2')= x_{\gamma_{\underline{l}}}(-j) \dots$ instead. Now factor $x_{\gamma_{\underline{\ell}}}(-j-1)$ goes on $x(\pi_1')=x_{\gamma_{\ell}}(-j) \dots $. We proceed in the same way; $x_{\gamma_{\underline{\ell}}}(-j-2)$ goes on $x_{\gamma_{\ell}}(-j-1) x(\pi_2')$, $x_{\gamma_{\ell}}(-j-2)$ on $x_{\gamma_{\underline{\ell}}}(-j-1)x(\pi'_1)$ and so on until we settle the $x_{\gamma_{\ell}}(-m-1)$ factor. Now we continue to take every second factor and add them on the partitions we have.
\begin{example}
Let $\ell = 6$ and let 
\begin{align*}
x(\pi) = & x_{\gamma_3}(-7) x_{\gamma_{\underline{5}}}(-6) x_{\gamma_{\underline{4}}}(-5)x_{\gamma_{6}}(-5) \\ & x_{\gamma_{\underline{6}}}(-4)x_{\gamma_{6}}(-4)x_{\gamma_{\underline{6}}}(-3) x_{\gamma_6}(-3)x_{\gamma_{\underline{6}}}(-2) x_{\gamma_6}(-2)x_{\gamma_5}(-1) x_{\gamma_3}(-1).
\end{align*}
We first factorize $x_{\gamma_{\underline{6}}}(-2) x_{\gamma_6}(-2)x_{\gamma_5}(-1) x_{\gamma_3}(-1)$ by taking every second term and so we get the subarrays:
\begin{align*}
 x_{\gamma_6}(-2)x_{\gamma_3}(-1) \quad \textrm{and} \quad x_{\gamma_{\underline{6}}}(-2) x_{\gamma_5}(-1) 
\end{align*}
Now we put $x_{\gamma_6}(-3)$ on the second subarray, i.e. $x_{\gamma_{\underline{6}}}(-2) x_{\gamma_5}(-1)$ (instead on the first) whilst $x_{\gamma_{\underline{6}}}(-3)$ goes on the first subarray, i.e. $x_{\gamma_6}(-2)x_{\gamma_3}(-1)$ (instead on the second). We continue in the same way, we place $x_{\gamma_6}(-4)$ on $x_{\gamma_{\underline{6}}}(-3)x_{\gamma_6}(-2)x_{\gamma_3}(-1)$ and $x_{\gamma_{\underline{6}}}(-4)$ on $x_{\gamma_6}(-3)x_{\gamma_{\underline{6}}}(-2)x_{\gamma_5}(-1)$. When we place all the factors with $\gamma_{6}, \gamma_{\underline{6}}$ colors we get the subarrays:
\begin{align*}
 & x_{\gamma_6}(-4)x_{\gamma_{\underline{6}}}(-3)x_{\gamma_6}(-2)x_{\gamma_3}(-1) \quad \textrm{and} \\ 
& x_{\gamma_6}(-5) x_{\gamma_{\underline{6}}}(-4)x_{\gamma_6}(-3)x_{\gamma_{\underline{6}}}(-2)x_{\gamma_5}(-1).
\end{align*}
Obviously, factor $x_{\gamma_{\underline{4}}}(-5)$ will be added to the subarray $x_{\gamma_6}(-4)x_{\gamma_{\underline{6}}}(-3)x_{\gamma_6}(-2)x_{\gamma_3}(-1) $. Now we continue to arrange the factors by taking every second term again, i.e. we finally get the following factorization:
\begin{align*}
 & x_{\gamma_3}(-7) x_{\gamma_{\underline{4}}}(-5) x_{\gamma_6}(-4)x_{\gamma_{\underline{6}}}(-3)x_{\gamma_6}(-2)x_{\gamma_3}(-1) \quad \textrm{and} \\ 
& x_{\gamma_{\underline{5}}}(-6) x_{\gamma_6}(-5) x_{\gamma_{\underline{6}}}(-4)x_{\gamma_6}(-3)x_{\gamma_{\underline{6}}}(-2)x_{\gamma_5}(-1).
\end{align*}
\end{example}
All other blocks from \eqref{iznimke} are handled in the similar way.

\end{proof}

\subsection{Initial conditions}\label{poc_uvj_2}
To establish the initial conditions (IC in short) for level 2 fundamental modules, we divide our modules in two groups. First group will be the level two modules $L(\Lambda)$ where weight $\Lambda$ is of the form $\Lambda = \Lambda^1 + \Lambda^2$ and $\Lambda^1, \Lambda^2 \in \{ \Lambda_0, \Lambda_1, \lam, \llam \}$, i.e. $\Lambda$ is a sum of two level 1 fundamental weights. Second group will be $L(\Lambda_j)$ modules for $j \in \{2, \dots, \ell-2 \}$.

We start with the first group. Let $W(\Lambda) = W(2\Lambda_0)$. We say that a monomial $x(\pi)$ satisfies \textit{the initial conditions for $W(2 \Lambda_0)$} if:
\begin{itemize}
\item[1)] $b_{\underline{3}} \dots + b_{\underline{\ell}} + b_{\ell} +
\dots + b_2  \leq 2$,
\item[2)] $b_{\underline{2}} + b_{\underline{3}} \dots +
b_{\underline{\ell}} + b_{\ell} + \dots + b_{3}  \leq 2$,
\end{itemize}
where $b_i$ (resp. $b_{\underline{i}}$) denotes the number of $x_{\gamma_i}(-1)$ (resp. $x_{\gamma_{\underline{i}}}(-1)$) factors. Observe that the initial conditions can be again memorized as the difference conditions, $DC$, in a sense that we don't add any imaginary factors of degree zero (see Remark \ref{dc_ic}).

Next proposition shows that this choice has an important property. 
\begin{proposition}\label{razap_poc}
The set of monomials
\begin{align}
\{ \ x(\pi) \ \vert \ x(\pi) \ \textrm{satisfies $DC$ and $IC$ for $W(2\Lambda_0)$} \ \}
\end{align}
spans $W(2\Lambda_0)$.
\end{proposition}
\begin{proof}
This is, in fact, already shown. Monomials which contain the leading terms can be replaced with a sum of greater ones (as in level 1) due to the compatibility of our ordering with multiplication in $U(\gtl_1)$ (see Proposition \ref{uredjaj}).
\end{proof}
We will also need the following lemma:
\begin{lemma} \label{rastav}
Monomial $x(\pi)$ satisfies IC for $W(2 \Lambda_0)$ if and only if its $(-1)$ part has a partition on two subarrays that satisfy IC for $W(\Lambda_0)$.
\end{lemma}

\begin{proof}
This can be deduced simply by looking at the $(-1)$ parts that satisfy IC conditions for $W(2 \Lambda_0)$. They
\begin{itemize}
 \item [-] either have only two $(-1)$ factors in which case we can put them separately on each $v_{\Lambda_0}$,
\item[-] or are of the form $x_{\gamma_{\underline{2}}}(-1)x_{\delta}(-1)x_{\gamma_2}(-1)$, or \\ $x_{\gamma_{\underline{2}}}(-1)x_{\gamma_{\underline{2}}}(-1)x_{\gamma_2}(-1)x_{\gamma_2}(-1)$. In both cases we remember that block $x_{\gamma_{\underline{2}}}(-1)x_{\gamma_2}(-1)$ satisfies $IC$ for $W(\Lambda_0)$ and this ensures the existence of the required partitions.
\end{itemize}
\end{proof}
Let now $W(\Lambda) = W(\Lambda^{1} + \Lambda^{2}) \neq W(2 \Lambda_0)$. We say that a monomial $x(\pi)$ satisfies \textit{the initial conditions for $W(\Lambda)$} if 
\begin{itemize}
 \item [1)] $b_{\ell-1} + \dots + b_2  \leq 2 - k_{\ell-1} - k_{\ell}- k_1 = k_0$ for $r=\ell-1$ in relation 1),
 \item [2)] $b_{\ell} + b_{\ell-1} + \dots + b_2  \leq 2 - k_{\ell} - k_1 = k_0 + k_{\ell-1}$ from relation 3),
\item[3)] $b_{\underline{\ell}} + b_{\ell-1} + \dots + b_2  \leq 2 - k_{\ell} - k_1 = k_0 + k_{\ell}$ from relation 4),
\item[4)] $b_{\underline{3}} + \dots + b_{\underline{\ell}} + b_{\ell} + \dots + b_2 \leq 2 - k_1 = k_0 + k_{\ell-1} + k_{\ell}$ for $r=2$ in relation 5),
\item[5)] $b_{\underline{2}} + \dots + b_{\underline{\ell}} + b_{\ell} + \dots + b_3 \leq 2 - k_1 = k_0 + k_{\ell-1} + k_{\ell}$ for $r=2$ in relation 6).
\end{itemize}
where $b_i$ (resp. $b_{\underline{i}}$) denotes the number of $x_{\gamma_i}(-1)$ (resp. $x_{\gamma_{\underline{i}}}(-1)$) factors in monomial $x(\pi)$.

We can again associate this initial conditions with the difference conditions via imaginary factors of degree zero (see Remark \ref{dc_ic}). If we apply the rules we established for the adding of these factors, we get the inital conditions from the difference conditions as before.

We have the following lemma:
 \begin{lemma} \label{part}
Let $\Lambda^1, \Lambda^2 \in \{\Lambda_0, \lam, \llam, \Lambda_1 \}$ and $\Lambda = \Lambda^1 + \Lambda^2$. A monomial $x(\pi)$ satisfies $IC$ for $W(\Lambda)$ if and only if its $(-1)$ part has a partition on two subarrays which satisfy $IC$ for $W(\Lambda^1)$ resp. $W(\Lambda^2)$.
\end{lemma}
\begin{proof}
Let $x(\pi)$ be a monomial that satisfies $IC$ for $W(\Lambda^1 + \Lambda^2)$. We can link the $IC$ and $DC$ conditions by adding the imaginary factors of degree zero. Therefore it is clear that $(-1) \dots (-0)$ part of the $x(\pi)$ monomial has a partition on subarrays that satisfy $DC$ conditions for level one. We now simply remove the added $(-0)$ part and get the desired partitions on subarrays that satisfy $IC$ conditions for $W(\Lambda^1)$ resp. $W(\Lambda^2)$. The inverse is trivial.
\end{proof}

Now we prove the analogue of Proposition \ref{razap_poc}:
\begin{proposition} \label{PU_UR_2}
 The set of monomials
\begin{align}
\{ \ x(\pi) \ \vert \ x(\pi) \ \textrm{satisfies $DC$ and $IC$ for $W(\Lambda) = W(\Lambda^{1} + \Lambda^{2})$} \ \} 
\end{align}
spans $W(\Lambda) = W(\Lambda^{1} + \Lambda^{2})$. 
\end{proposition}
\begin{proof}
Suppose first that $\Lambda^1, \Lambda^2 \neq \Lambda_1$ and let $x(\pi)$ be an arbitrary monomial. Then, by Proposition \ref{razap_poc}, $x(\pi)$ can be expressed as a linear combination of monomials that satisfy $DC$ and $IC$ on $W(2\Lambda_0)$, i.e. we have
\begin{align}
x(\pi) (v_{\Lambda_0} \otimes v_{\Lambda_0}) = \sum c_{\pi'}x(\pi')(v_{\Lambda_0} \otimes v_{\Lambda_0}).
\end{align}
Using the tensor product of the intertwining operators from level one (see Section \ref{operatori_1}) we map $v_{\Lambda_0} \otimes v_{\Lambda_0}$ to $v_{\Lambda^1} \otimes v_{\Lambda^2}$, i.e. to the highest weight vector for the $L(\Lambda^{1} + \Lambda^{2})$ module. This gives us:
\begin{align}\label{suma_PU}
x(\pi) (v_{\Lambda^1} \otimes v_{\Lambda^2}) = \sum c_{\pi'}x(\pi')(v_{\Lambda^1} \otimes v_{\Lambda^2}).
\end{align}
The used map is surjective (follows directly from the $IC$). We claim that in \eqref{suma_PU} $x(\pi')(v_{\Lambda^1} \otimes v_{\Lambda^2}) \neq 0$ if and only if $x(\pi')$ satisfies $DC$ and $IC$ on $W(\Lambda^{1} + \Lambda^{2})$. Since $DC$ are the same for all subspaces it remains only to check whether $x(\pi')$ satisfies $IC$. Let $x(\pi')( v_{\Lambda^1} \otimes v_{\Lambda^2}) \neq 0$. If $(-1)$ part of this monomial has a partition on subarrays that satisfy $IC$ for $W(\Lambda^1)$, resp. $W(\Lambda^2)$, then Lemma \ref{part} shows that $x(\pi')$ satisfies $IC$ for $W(\Lambda^1 + \Lambda^2)$ as well. Suppose therefore $x(\pi')$ has no such partition. This means $x(\pi')$ has a partition of $(-1)$ part that doesn't annihilate $v_{\Lambda^1}$ or $v_{\Lambda^2}$ but at the same time it doesn't satisfy $IC$ for at least one of the $W(\Lambda^i)$, $i=1,2$, either. Then this $\Lambda^i$, suppose it's $\Lambda^1$, has to be $\Lambda_0$. Namely, $x_{\delta}(-1)v_{\Lambda'} \neq 0$ implies that $x_{\delta}(-1)$ satisfies $IC$ for $W(\Lambda')$ if $\Lambda' \neq \Lambda_0$ (see Subsection \ref{pocetni_uvjeti}). So, for $v_{\Lambda^1} = v_{\Lambda_0}$ we have $x_{\underline{\gamma}}(-1)x_{\gamma}(-1)$, $\gamma \neq \gamma_2$. At the same time, we have factor $x_{\tau}(-1)$ which will be in this partition put on vector  $v_{\Lambda^2}$. This means the $(-1)$ part of $x(\pi')$ looks as follows: $x_{\underline{\gamma}}(-1)x_{\gamma}(-1)x_{\tau}(-1)$, $\gamma \neq \gamma_2$, which is a contradiction with the assumption that $x(\pi')$ satisfies $IC$ on $W(2 \Lambda_0)$.

Let now one of the vectors $v_{\Lambda^i}$ be equal to $v_{\Lambda_1}$, say $v_{\Lambda^1} = v_{\Lambda_1}$. Then $IC$ look the same as the $IC$ for the level 1 vector $v_{\Lambda'}=v_{\Lambda^2}$ which we get if we remove $v_{\Lambda_1}$ from $v_{\Lambda} = v_{\Lambda_1} \otimes v_{\Lambda^2}$. Therefore the assertion follows from Proposition \ref{razap}.

If we have $v_{\Lambda} = v_{\Lambda_1} \otimes v_{\Lambda_1}$, the theorem is obvious since no $(-1)$ parts are allowed, as the $IC$ define.
\end{proof}

We now analyze the second group of level 2 modules. i.e. $W(\Lambda_j)$, $j=2, \dots, \ell-2$.  
We say that a monomial $x(\pi)$ satisfies \textit{the initial conditions (or $IC$ in short) for $W(\Lambda_j)$} if:
\begin{itemize}
 \item [1)] $b_{j} + b_{j-1} + \dots + b_2  \leq 0$ 

\item [2)] $b_{j+1} + b_{j} + \dots + b_2  \leq 1$,
\item[3)] $b_{\ell} + b_{\ell-1} + \dots + b_2  \leq 1$,
\item[4)] $b_{\underline{\ell}} + b_{\ell-1} + \dots + b_2 \leq 1$,
\item[5)] $b_{\underline{j+2}} + b_{\underline{j+3}} + \dots + b_{\underline{\ell}} + b_{\ell} + \dots + b_2 \leq 1$,
\item[6)] $b_{\underline{j+1}} + b_{\underline{j+2}} + \dots + b_{\underline{\ell}} + b_{\ell} + \dots + b_{j+2} + b_j + \dots + b_2 \leq 1$,
\item[7)] $b_{\underline{m+1}} +  \dots + b_{\underline{\ell}} + b_{\ell} + \dots +  b_2 \leq 2$, for $m < j+1$,
\item[6)] $b_{\underline{m}} + b_{\underline{m+1}} + \dots + b_{\underline{\ell}} + b_{\ell} + \dots + b_{m+1} + b_m + \dots + b_2 \leq 2$ for $m < j+1$.
\end{itemize}

where $b_i$ (resp. $b_{\underline{i}}$) denotes the number of $x_{\gamma_i}(-1)$ (resp. $x_{\gamma_{\underline{i}}}(-1)$) factors in this monomial.

We can memorize this $IC$ via $DC$ as well. For $W(\Lambda_j)$, $j=2, \dots, \ell-2$, we add the imaginary $x_{\gamma_{\underline{j+1}}}(-0)x_{\gamma_{j+1}}(-0)$ factors to the $DC$ for level 2, i.e. we set $a_{j+1} = a_{\underline{j+1}} = 1$ in \eqref{uvj_razl_dr_zap} whilst all the others $a_i$ remain zero. Then the $b_i$ (resp. $b_{\underline{i}}$), $i=2, \dots, \ell$ satisfy the $IC$ for subspace $W(\Lambda_j)$. We find the motivation for this rule in the following fact: we can memorize the $IC$ for $W(\lam + \llam)$ by adding the $x_{\gamma_{\underline{\ell}}}(-0)x_{\gamma_{\ell}}(-0)$. Further on, $x_{\gamma_{\underline{\ell}}}(-1)x_{\gamma_{\ell}}(-1)(v_{\lam} \otimes v_{\llam}) = C \cdot e^{\omega}(v_{\lam} \otimes v_{\llam}) $. We now use the relation $x_{\gamma_{\underline{j+1}}}(-1)x_{\gamma_{j+1}}(-1)v_{\Lambda_j} = C \cdot e^{\omega}v_{\Lambda_j}$ (see Remark \ref{e_2}) and see that the corresponding imaginary factors for $W(\Lambda_j)$ are $x_{\gamma_{\underline{j+1}}}(-0)$ and $x_{\gamma_{j+1}}(-0)$.
\vspace{2mm}

The $IC$ for $W(\Lambda_j)$, $j=2, \dots, \ell-2$, can be described in a different way:
\begin{lemma} \
\begin{itemize}
 \item [a)] If $x_{\gamma}(-1)$ satisfies $IC$ for $W(\Lambda_j)$ then $\gamma < \gamma_j$.
 \item [b)] If $x_{\tau}(-1)x_{\delta}(-1)$ satisfy $IC$ for $W(\Lambda_j)$ then $\delta \leq \gamma_{j+1}$, $\tau \leq \gamma_{\underline{j+1}}$ and if $\tau = \gamma_{\underline{j+1}}$ then necessarily $\delta = \gamma_{j+1}$.
\end{itemize}
Also the following holds:
\begin{align}
\label{po_uvjet_2} x_{\underline{\delta}}(-1)x_{\delta}(-1)v_{\Lambda_j} &= C_1 \cdot  x_{\gamma_{\underline{j+1}}}(-1)x_{\gamma_{j+1}}(-1)v_{\Lambda_j} = C_2 \cdot e^{\omega} v_{\Lambda_j}, \\ \label{po_uvjet_23} x_{\tau}(-1)x_{\delta}(-1)v_{\Lambda_j} &= 0 \quad \textrm{if $\tau \neq \underline{\delta}$ and $\tau, \delta$ don't satisfy the conditions in b)} .
\end{align}
\end{lemma}
\begin{proof}
Claims a) and b) follow directly from the definition of the $IC$. Assertion \eqref{po_uvjet_2} is also clear; basic vectors are collinear in $L(\Lambda_j)$ (see Proposition \ref{kolinearnost}). Namely, using Lemma \ref{lem_op_1} we get (suppose $\delta = \gamma_i$ and $i \in \Psi_2$):
$$x_{\underline{\delta}}(-1)x_{\delta}(-1)(w_{12 \dots j \Psi_1} \otimes w_{12 \dots j \Psi_2}) = C \cdot e^{\omega} (w_{12 \dots j \Psi_1'} \otimes w_{12 \dots j \Psi_2'}) $$
where $\Psi_1' = \Psi_1 \cup \{ i \}$ and $\Psi_2' = \Psi_2 - \{ i \}$. So, $x_{\underline{\delta}}(-1)x_{\delta}(-1)$ transferred one basic vector into $e^{\omega}$ multiplied by another basic vector. 

Now we show \eqref{po_uvjet_23}. Let $x_{\tau}(-1)x_{\delta}(-1)$ be such that $\tau \neq \underline{\delta}$ and such that $\tau, \delta$ don't satisfy the conditions in b). We have the following subcases:
\begin{itemize}
 \item [-] $\delta = \gamma_i \geq \gamma_j$. Lemma \ref{lem_op_1} gives us $x_{\delta}(-1)(w_{12 \dots j \Psi_1} \otimes w_{12 \dots j \Psi_2}) = 0$ since $i \in \Psi_1$ and $i \in \Psi_2$.
\item [-] $\delta = \gamma_{j+1}$, $\tau = \gamma_m > \gamma_{\underline{j+1}}$. We take the basic vector $w_{12 \dots j \Psi_1} \otimes w_{12 \dots j \Psi_2}$ such that $\{ m , j+ 1\} \subset \Psi_1$ as a representative of $v_{\Lambda_j}$. Then  $x_{\delta}(-1)w_{12 \dots j \Psi_1} = x_{\tau}(-1)w_{12 \dots j \Psi_1} =0$.
\item [-] $\delta = \gamma_{j+1}$, $\tau = \gamma_{\underline{m}} > \gamma_{\underline{j+1}}$. We take the basic vector $w_{12 \dots j \Psi_1} \otimes w_{12 \dots j \Psi_2}$ such that $j+ 1 \in \Psi_1$ and $m \notin \Psi_1$ as a representative of $v_{\Lambda_j}$ and again we have $x_{\delta}(-1)w_{12 \dots j \Psi_1} = x_{\tau}(-1)w_{12 \dots j \Psi_1} =0$.
\item [-] $\delta < \gamma_{j+1}$, $\tau \geq \gamma_{\underline{j+1}}$. Since $\tau \neq \underline{\delta}$ we can again choose the basic vector, a representative of $v_{\Lambda_j}$, such that $x_{\delta}(-1)w_{12 \dots j \Psi_1} = x_{\tau}(-1)w_{12 \dots j \Psi_1} =0$.
\end{itemize}
\end{proof}
The lemma above helps us to prove the next proposition:
\begin{proposition} \label{PU_UR_22}
 Let $j \in \{2, \dots, \ell-2 \}$. The set of vectors
\begin{align}
\{ \ x(\pi)v_{\Lambda_j} \ \vert \ x(\pi) \ \textrm{satisfies $DC$ and $IC$ for $W(\Lambda_j)$} \ \} 
\end{align}
spans $W(\Lambda_j)$.
\end{proposition}
\begin{proof}
Similarly to the way it is done in Proposition \ref{PU_UR_2}. Let $x(\pi)$ be a monomial such that  $x(\pi)v_{\Lambda_j} \neq 0$. By Proposition \ref{razap_poc} we have:
\begin{align}
x(\pi) (v_{\Lambda_0} \otimes v_{\Lambda_0}) = \sum c_{\pi'}x(\pi')(v_{\Lambda_0} \otimes v_{\Lambda_0})
\end{align}
where all monomials $x(\pi')$ satisfy $DC$ and $IC$ conditions on $W(2\Lambda_0)$.
Now we send, using the intertwining operators, $v_{\Lambda_0} \otimes v_{\Lambda_0}$ to some basic vector and then using the quotient map, this basic vector to $v_{\Lambda_j}$. We get 
\begin{align} \label{suma_j}
x(\pi) v_{\Lambda_j} = \sum c_{\pi'}x(\pi') v_{\Lambda_j}.
\end{align}
Our claim is that in sum \eqref{suma_j} ``survive" only the monomials $x(\pi')$ that 
\begin{itemize}
 \item [a)] either satisfy $DC$ and $IC$ for $W(\Lambda_j)$,
\item [b)] or they don't satisfy $DC$ and $IC$ for $W(\Lambda_j)$ but can be replaced with monomials that do.
\end{itemize}
Previous lemma shows us that the only monomials which don't satisfy the $IC$ for $W(\Lambda_j)$ such that $x(\pi')v_{\Lambda_j} \neq 0$, are of the form $x_{\gamma_{\underline{m}}}(-1)x_{\gamma_m}(-1)$ where $m < j+1$. Equation \eqref{po_uvjet_2} implies that this $(-1)$ part can be replaced by $x_{\gamma_{\underline{j+1}}}(-1)x_{\gamma_{j+1}}(-1)$. Since $x(\pi')v_{\Lambda_j} \neq 0$, the only case in which this new monomial won't satisfy the $DC$ is when $x(\pi')$ is of the form: $$x(\pi') = \dots x_{\gamma_{\underline{m}}}(-2)x_{\gamma_m}(-2)x_{\gamma_{\underline{m}}}(-1)x_{\gamma_m}(-1).$$ But then:
\begin{align*}
 &x_{\gamma_{\underline{m}}}(-2)x_{\gamma_m}(-2)x_{\gamma_{\underline{m}}}(-1)x_{\gamma_m}(-1)v_{\Lambda_j} =  x_{\gamma_{\underline{m}}}(-2)x_{\gamma_m}(-2) \cdot C \cdot e^{\omega} v_{\Lambda_j} = \\ & = C' \cdot e^{\omega } x_{\gamma_{\underline{m}}}(-1)x_{\gamma_m}(-1)v_{\Lambda_j} = C'' \cdot e^{2 \omega} v_{\Lambda_j}
\end{align*}
and similarly
$$x_{\gamma_{\underline{j+1}}}(-2)x_{\gamma_j}(-2)x_{\gamma_{\underline{j+1}}}(-1)x_{\gamma_j}(-1)v_{\Lambda_j} = D'' \cdot e^{2 \omega} v_{\Lambda_j}.$$

So, we can replace the whole $x_{\gamma_{\underline{m}}}(-2)x_{\gamma_m}(-2)x_{\gamma_{\underline{m}}}(-1)x_{\gamma_m}(-1)$ block with the one that satisfies the $IC$ for $\Lambda_j$. If $DC$ are now still violated this means we have:
$$x(\pi') = \dots x_{\gamma_{\underline{m}}}(-3)x_{\gamma_m}(-3) x_{\gamma_{\underline{m}}}(-2)x_{\gamma_m}(-2)x_{\gamma_{\underline{m}}}(-1)x_{\gamma_m}(-1)$$
in which case we proceed analogously. 
\end{proof}

\begin{remark}
For a level 2 module $L(\Lambda)$ we will call the vectors that satisfy $DC$ and $IC$ for Feigin-Stoyanovsky's subspace $W(\Lambda)$ \textit{the admissible vectors} as before.
\end{remark}

\section{Basic intertwining operators and level 2 modules}
In this section we will establish the result analogous to Proposition \ref{operatori} for level 2 modules for algebra of type $D_4^{(1)}$.

\begin{remark}\label{najgori}
In the proof of the following proposition we will use the term "the worst possible case". This term refers to the monomials which have a small difference in the degrees of their successive factors (i.e. their degrees usually decline by one). To prove the claim then actually means to analyze exactly such monomials since the rest follows trivially.
\end{remark}

\begin{proposition}\label{operatori21}
Let $L(\Lambda^1 + \Lambda^2)$, $\Lambda^1, \Lambda^2 \in \{\Lambda_0, \Lambda_1, \lam, \llam \}$ be a level two module for algebra $D_4^{(1)}$ and $v_{\Lambda} = v_{\Lambda^1} \otimes v_{\Lambda^2}$ its highest weight vector. Let $x(\pi)$ be a monomial that satisfies $DC$ and $IC$ for $W(\Lambda^1 + \Lambda^2)$. Then there exist a factorization of monomial $x(\pi)$, $x(\pi) = x(\pi_2) x(\pi_1)$ and an intertwining operator $I_{\pi_1, \Lambda}$, which commutes with the action of $\gtl_1$, so that the following holds:
\begin{itemize}
 \item[a)] $I_{\pi_1, \Lambda}x(\pi_1)v_{\Lambda} = C \cdot e^{n \omega} v_{\Lambda'} \neq 0$ for some $n$ and some $\Lambda'$,
 \item[b)] $I_{\pi_1, \Lambda}x(\pi_1')v_{\Lambda} = 0$ for $x(\pi_1') > x(\pi_1)$,
 \item[c)] $I_{\pi_1, \Lambda}x(\pi_1)v_{\Lambda}= C \cdot e^{n \omega} v_{\Lambda'}$ and $x(\pi_2)^{+n}$ satisfies $DC$ and $IC$ on $W(\Lambda')$.
\end{itemize}
\end{proposition}

\begin{proof}
By Proposition \ref{particija} and Lemma \ref{rastav} we know that $DC$ and $IC$ for $W(\Lambda)$ can be described as follows:
\begin{itemize}
 \item [a)] either a monomial $x(\pi)$ has a partition on two subarrays $x(\pi^1)$ and $x(\pi^2)$ where $x(\pi^1)$, resp. $x(\pi^2)$, satisfies $DC$ and $IC$ on $W(\Lambda^1)$, resp. $W(\Lambda^2)$,
\item [b)] or this monomial satisfies the exceptional $DC$, i.e. contains the blocks of the form $x_{\gamma_3}(-(j+1))x_{\gamma_{\underline{3}}}(-j)x_{\gamma_3}(-j)$, resp. $x_{\gamma_{\underline{3}}}(-(j-1))x_{\gamma_3}(-(j-1))x_{\gamma_{\underline{3}}}(-j)$ (see Proposition \ref{uv_raz_2}).
\end{itemize}
We will prove the existence of the required factorizations and operators case wise.
\vspace{2mm}

\textbf{a):} We know, by Proposition \ref{operatori}, that for the partitions $x(\pi^1)$ and $x(\pi^2)$ we have the factorization $x(\pi^1)=x(\pi^1_2)x(\pi^1_1)$, resp. $x(\pi^2)=x(\pi^2_2)x(\pi^2_1)$ and basic operators $I_{\pi^1_1, \Lambda^1}$, resp. $I_{\pi^2_1, \Lambda^2}$, such that:
\begin{itemize}
 \item [i)] if $x(\pi') > x(\pi_1^1)$ (resp. $x(\pi') > x(\pi_1^2)$), then $I_{\pi^1_1, \Lambda^1}x(\pi')v_{\Lambda^1} =0$ (resp. $I_{\pi^2_1, \Lambda^2}x(\pi')v_{\Lambda^2} =0$),
\item [ii)] $I_{\pi_1^1}x(\pi^1_1)v_{\Lambda^1} = C \cdot e^{n \omega} v_{\Lambda'}$ and $I_{\pi_1^2}x(\pi^2_1)v_{\Lambda^2} = C \cdot e^{m \omega} v_{\Lambda''}$.
\end{itemize}
Now we form the tensor product of this operators and we have:
\begin{align} \label{izvlacenje}
& I_{\pi^1_1,\Lambda^1} \otimes I_{\pi^2_1, \Lambda^2} (x(\pi) (v_{\Lambda^1} \otimes v_{\Lambda^2})) = \\ \notag &= I_{\pi^1_1,\Lambda^1} \otimes I_{\pi^2_1, \Lambda^2}  \left(   \sum_{\substack{\textrm{all factorizations of} \\  \textrm{$x(\pi) = x(\mu^2) x(\mu^1)$}}}  x(\mu^1)v_{\Lambda^1} \otimes x(\mu^2)v_{\Lambda^2} \right)  = \\ \notag &= I_{\pi^1_1,\Lambda^1} \otimes I_{\pi^2_1, \Lambda^2} \left( \sum \dots  x(\pi^1_1)v_{\Lambda^1} \otimes \dots  x(\pi^2_1)v_{\Lambda^2} + \sum_{\substack{\textrm{other factorizations} \\ \textrm{of $x(\pi)$}}} x(\mu^1)v_{\Lambda^1} \otimes x(\mu^2 )v_{\Lambda^2}  \right) = \\ \notag & = \sum \dots  I_{\pi^1, \Lambda^1} x(\pi^1_1)v_{\Lambda^1} \otimes \dots I_{\pi^2,\Lambda^2} x(\pi^2_1)v_{\Lambda^2} = 
\sum \dots C_1 \cdot e^{n_1 \omega} v_{\Lambda^{11}} \otimes \dots C_2 \cdot e^{m_1 \omega} v_{\Lambda^{21}}, 
\end{align}
since operator $I_{\pi^1_1,\Lambda^1} \otimes I_{\pi^2_1, \Lambda^2}$ annihilates all factorizations except the ones of the form $\dots  x(\pi^1_1)v_{\Lambda^1} \otimes \dots  x(\pi^2_1)v_{\Lambda^2}$ (this follows directly from the property ii) and the compatibility of the multiplication with the ordering). 

In case $n_1=m_1$ we have:
$$ \dots C_1 \cdot e^{n_1 \omega} v_{\Lambda^1   \hspace{0.2mm} '} \otimes \dots C_2 \cdot e^{n_1 \omega} v_{\Lambda^2  \hspace{0.2mm} '} = \dots D \cdot e^{n_1 \omega} (v_{\Lambda^1  \hspace{0.2mm} '} \otimes v_{\Lambda^2  \hspace{0.2mm} '}).$$
If we set $x(\pi_1) = x(\pi^1_1)x(\pi^2_1)$ and $I_{\pi_1,\Lambda} = I_{\pi^1_1,\Lambda^1} \otimes I_{\pi^2_1, \Lambda^2}$ we obviously get the factorization and the operator required in the proposition. Here also $v_{\Lambda'} = v_{\Lambda^1 \hspace{0.2mm} '} \otimes v_{\Lambda^2  \hspace{0.2mm} '}$. 

What if $m_1 \neq n_1$? For monomials $x(\pi^1_2)$ and $x(\pi^2_2)$ and vectors $e^{n_1 \omega} v_{\Lambda^{11}}$, $e^{m_1 \omega} v_{\Lambda^{21}}$ we can find factorizations $x(\pi^1_2) = x(\pi^1_{22}) x(\pi^1_{21})$, $x(\pi^2_2) = x(\pi^2_{22})x(\pi^2_{21})$ and the operators $I_{\pi^1_{21}, \Lambda^{11}}$, $I_{\pi^2_{21},\Lambda^{21}}$ such that
\begin{itemize}
 \item [i)] if $x(\pi') > x(\pi_{21}^1)$ (resp. $x(\pi') > x(\pi_{21}^2)$), then $I_{\pi^1_{21}, \Lambda^{11}} x(\pi')  e^{n_1 \omega} v_{\Lambda^{11}} =0$ (resp. $I_{\pi^2_1, \Lambda^{21}}x(\pi') e^{m_1 \omega} v_{\Lambda^{21}} =0$),
\item [ii)] $I_{\pi^1_{21}, \Lambda^{11}} x(\pi^1_{21})  e^{n_1 \omega} v_{\Lambda^{11}} = D_1 \cdot e^{n_2 \omega} v_{\Lambda^{12}}  $ \\ (resp. $I_{\pi^1_{21}, \Lambda^{21}} x(\pi^2_{21})  e^{m_2 \omega} v_{\Lambda^{21}} = D_2 \cdot e^{m_2 \omega} v_{\Lambda^{22}}$).
\end{itemize}
(see Proposition \ref{operatori} and Remark \ref{kompon}). We apply operator $I_{\pi^1_{21}, \Lambda^{11}} \otimes I_{\pi^2_{21},\Lambda^{21}}$ on \eqref{izvlacenje} and get

\begin{align}
& I_{\pi^1_{21}, \Lambda^{11}} \otimes I_{\pi^2_{21},\Lambda^{21}} \left( \sum \dots C_1 \cdot e^{n_1 \omega} v_{\Lambda^{11}} \otimes \dots C_2 \cdot e^{m_1 \omega} v_{\Lambda^{21}} \right) = \\ \notag & I_{\pi^1_{21}, \Lambda^{11}} \otimes I_{\pi^2_{21},\Lambda^{21}} \left( \sum \dots C_1 \cdot x(\pi^1_{21}) e^{n_1 \omega} v_{\Lambda^{11}} \otimes \dots C_2 \cdot x(\pi^2_{21}) e^{m_1 \omega} v_{\Lambda^{21}} \right) + \\ \notag & + I_{\pi^1_{21}, \Lambda^{11}} \otimes I_{\pi^2_{21},\Lambda^{21}} \left(  \sum_{\substack{\textrm{other factorizations of} \\ \textrm{ $ \dots x(\pi^1_2) x(\pi^2_2)$}}} \dots C_1 \cdot e^{n_1 \omega} v_{\Lambda^{11}} \otimes \dots C_2 \cdot e^{m_1 \omega} v_{\Lambda^{21}} \right) = \\ = \notag &  \sum \dots C_1 \cdot I_{\pi^1_{21}, \Lambda^{11}} x(\pi^1_{21}) e^{n_1 \omega} v_{\Lambda^{11}} \otimes \dots C_2 \cdot I_{\pi^2_{21},\Lambda^{21}} x(\pi^2_{21}) e^{m_1 \omega} v_{\Lambda^{21}} = \\ = \notag &  \sum \dots D_1 \cdot e^{n_2 \omega} v_{\Lambda^{12}} \otimes \dots D_2 \cdot  e^{m_2 \omega} v_{\Lambda^{22}}
\end{align}

since $I_{\pi^1_{21}, \Lambda^{11}} \otimes I_{\pi^2_{21},\Lambda^{21}}$ annihilates all other factorizations. Now, if $n_2=m_2$ we have:
\begin{align}
 \sum \dots D_1 \cdot e^{n_2 \omega} v_{\Lambda^{12}} \otimes \dots D_2 \cdot  e^{n_2 \omega} v_{\Lambda^{22}} = \sum \dots D \cdot e^{n_2 \omega} (v_{\Lambda^{12}} \otimes  v_{\Lambda^{22}}).
\end{align}
We can now put $x(\pi_1) =  x(\pi^1_{21}) x(\pi^2_{21}) x(\pi^1_1)x(\pi^2_1)$ and define 
$I_{\pi_1, \Lambda}$ as $$(I_{\pi^1_{21}, \Lambda^{11}} \circ I_{\pi^1_1, \Lambda_1}) \otimes (I_{\pi^2_{21},\Lambda^{21}} \circ I_{\pi^2_1, \Lambda^2}).$$
We set $v_{\Lambda'} = v_{\Lambda^{12}} \otimes v_{\Lambda^{22}}$ and we get the factorization and the operator with the required properties from the proposition. 

If $n_2 \neq m_2$, we continue in the same manner until we get the equality or reach the end of one of the monomials $x(\pi^1)$, $x(\pi^2)$. In this other case (let's say it happens after $i$ steps and suppose we reached the end of $x(\pi^1)$) we will have:
\begin{align}
& \sum C_i \cdot I_{\pi^1_{i1}, \Lambda^{1(i-1)}} x(\pi^1_{i1}) e^{n_i \omega} v_{\Lambda^{1(i-1)}} \otimes \dots C'_i \cdot I_{\pi^2_{i1},\Lambda^{2(i-1)}} x(\pi^2_{i1}) e^{m_i \omega} v_{\Lambda^{2(i-1)}} = \\ = \notag &  \sum D_1 \cdot e^{n_i \omega} v_{\Lambda^{1i}} \otimes \dots C'_i \cdot  e^{m_i \omega} v_{\Lambda^{2i}}.
\end{align}
In case $n_i = m_i$, we are done. If $n_i < m_i$ we can use the basic skip operator to map $e^{n_i \omega} v_{\Lambda^{1i}}$ to $e^{m_i \omega} v_{\Lambda^{1i}}$ (see \eqref{op4}) and so get the desired equalitiy. If $n_i > m_i$ we simply use the next operator for the second component, i.e. $I_{\pi^2_{(i+1)1}, \Lambda^{2 i}}$ whilst on the first component we use the identity, $Id$. We get:
\begin{align}
& \sum C_i \cdot Id \ e^{n_i \omega} v_{\Lambda^{1 i}} \otimes \dots C_i' \cdot I_{\pi^2_{(i+1)1}, \Lambda^{2 i}} x(\pi^2_{(i+1)1}) e^{m_i \omega} v_{\Lambda^{2i}} = \\ = \notag &  \sum C_i \cdot  e^{n_i \omega} v_{\Lambda^{1i}} \otimes \dots C_{i+1}' \cdot  e^{m_{i+1} \omega} v_{\Lambda^{2(i+1)}}.
\end{align}
Now we compare $n_i$ and $m_{i+1}$ and proceed in an analogous way if necessary.
\vspace{2mm}

\textbf{b):} Let $x(\pi)$ be a monomial of the form $x(\pi) = \dots x_{\gamma_{\underline{3}}}(-m)x_{\gamma_3}(-m)\dots$ where $x_{\gamma_{\underline{3}}}(-m)x_{\gamma_3}(-m)$ are the first elements which satisfy the exceptional $DC$. This means $x(\pi)$ belongs to the one of the following cases:
\begin{itemize}
 \item [i)] $x(\pi) = \dots x_{\gamma_{\underline{3}}}(-m)x_{\gamma_3}(-m)x_{\gamma_{\underline{3}}}(-(m-1))\dots$ or
\item [ii)]$x(\pi) = \dots x_{\gamma_3} (-(m+1))x_{\gamma_{\underline{3}}}(-m)x_{\gamma_3}(-m)\dots$.
\end{itemize}
We now write $x(\pi) = x(\pi'')x_{\gamma_{\underline{3}}}(-m)x_{\gamma_3}(-m)x(\pi')$.

Since $x_{\gamma_{\underline{3}}}(-m)x_{\gamma_3}(-m)$ are the first elements which satisfy the exceptional $DC$, monomial $x(\pi')$ has a partition as in a), on two subarrays $x(\pi'_1)$ and $x(\pi_2')$ which satisfy $DC$ and $IC$ on $v_{\Lambda^1} \otimes v_{\Lambda^2}$. We get this subarrays, as before, by taking every second factor. 

Let now $v_{\Lambda_2}$ be a highest weight vector for $L(\Lambda_2)$. We know, by Proposition \ref{vektor_j}, that $v_{\Lambda_2}$ has a realization in tensor product of level 1 modules $L(\Lambda_3)$ and $L(\Lambda_4)$:
\begin{align}
v_{\Lambda_2} &= C_1 \cdot w_{123} \otimes w_{124} + C_2 \cdot w_{124} \otimes w_{123} \quad \textrm{in $L(\Lambda_3) \otimes L(\Lambda_3)$}, \\ v_{\Lambda_2} &= C_1 \cdot w_{12} \otimes w_{1234} + C_2 \cdot w_{1234} \otimes w_{12} \quad \textrm{in $L(\Lambda_4) \otimes L(\Lambda_4)$}.
\end{align}
When we regard $L(\Lambda_2)$ as a quotient of the tensor product, vectors $w_{123} \otimes w_{124}$ and  $w_{124} \otimes w_{134}$ (resp. $w_{12} \otimes w_{1234}$ and $w_{1234} \otimes w_{12}$) are collinear (see Proposition \ref{kolinearnost}).

Now we want to construct the intertwining operator $I_{\pi', \Lambda}$ with the following property:
\begin{align}
 \label{svojstvo1} I_{\pi',\Lambda}x(\pi')(v_{\Lambda^1} \otimes v_{\Lambda^2}) &= C \cdot e^{(m-1) \omega}v_{\Lambda_2} \quad \textrm{and} \\
 \label{svojstvo2} I_{\pi',\Lambda} x(\mu)(v_{\Lambda^1} \otimes v_{\Lambda^2}) &= 0 \quad \textrm{if } x(\mu)>x(\pi').
\end{align}

In the worst case (see Remark \ref{najgori}) we have:
\begin{align}
\label{izgled2} x(\pi_2') &= x_{\delta}(-(m-1)) \dots, \ \delta \geq \gamma_{\underline{3}} \quad \textrm{and} \\
\label{izgled1} x(\pi_1') &= x_{\gamma_2}(-(m-1)) \dots \quad \textrm{or} \quad x(\pi_1') &= x_{\delta}(-p) \dots \ \textrm{where $p < m-1$}.
\end{align}

We will define $I'_{\pi', \Lambda}$ as a tensor product of the operators $I_{\pi_1', \Lambda^1}$ and $I_{\pi_2', \Lambda^2}$. 
The intertwining operator $I_{\pi_2', \Lambda^2}$ acts on vector $x(\pi_2')v_{\Lambda^2}$ in the following way (see Proposition \ref{operatori}):
\begin{align}
\label{forma1} I_{\pi_2', \Lambda^2}x(\pi_2')v_{\Lambda^2} &= C \cdot e^{(m-1) \omega} w_{124} \quad \textrm{or} \\ \label{forma2} I_{\pi_2', \Lambda^2}x(\pi_2')v_{\Lambda^2} &= C \cdot e^{(m-1) \omega} w_{12} \quad \textrm{or} \\
\label{forma3}  I_{\pi_2', \Lambda^2}x(\pi_2')v_{\Lambda^2} &= C \cdot e^{(m-1) \omega} w_{123} \quad \textrm{or} \\
\label{forma4}  I_{\pi_2', \Lambda^2}x(\pi_2')v_{\Lambda^2} &= C \cdot e^{(m-1) \omega} w_{1234},
\end{align}
depending on the form of $x(\pi_2')$. For instance, if \\ $x(\pi'_2) = x_{\gamma_{\underline{3}}}(-(m-1))x_{\gamma_4}(-(m-2)) \dots$, then we'll have the case \eqref{forma1}, i.e. $I_{\pi_2', \Lambda^2}x(\pi_2')v_{\Lambda^2} = C \cdot e^{(m-1) \omega} w_{124}$.

Now we choose operator $I_{\pi_1, \Lambda^1}$ so that it acts on vector $x(\pi_1')v_{\Lambda^1}$ as 
\begin{align}
I_{\pi_1', \Lambda^1} x(\pi_1')v_{\Lambda^1} &= C \cdot e^{(m-1)} w_{123} \quad \textrm{in case of \eqref{forma1}}, \\ I_{\pi_1', \Lambda^1} x(\pi_1')v_{\Lambda^1} &= C \cdot e^{(m-1)} w_{1234} \quad \textrm{in case of \eqref{forma2}}, \\ I_{\pi_1', \Lambda^1} x(\pi_1')v_{\Lambda^1} &= C \cdot e^{(m-1)} w_{124} \quad \textrm{in case of \eqref{forma3}}, \\ I_{\pi_1', \Lambda^1} x(\pi_1')v_{\Lambda^1} &= C \cdot e^{(m-1)} w_{12} \quad \textrm{in case of \eqref{forma4}}.
\end{align}
This we can certainly do since $x(\pi_1')$ has form \eqref{izgled1} (see \eqref{operator11}-\eqref{operator13}). This means the operator $I'_{\pi', \Lambda} =  I_{\pi_1'} \otimes I_{\pi_2'}$ composed with the quotient mapping $q: L(\Lambda_3) \otimes L(\Lambda_3) \longrightarrow L(\Lambda_2)$ (resp. $q : L(\Lambda_4) \otimes L(\Lambda_4) \longrightarrow L(\Lambda_2)$) has the property \eqref{svojstvo1} when applied to $x(\pi_1')v_{\Lambda^1} \otimes x(\pi_2')v_{\Lambda^2}$. We will denote this operator with $I_{\pi', \Lambda}$, i.e. we set $I_{\pi', \Lambda} = q \circ I'_{\pi', \Lambda}.$

It remains to verify the property \eqref{svojstvo2}. In order to do so, we have to take a closer look to the partitions $x(\pi_1')$, $x(\pi_2')$ and the operators $I_{\pi_1'}$, $I_{\pi_2'}$. Let $x(\mu)$ be some monomial such that $x(\mu) > x(\pi')$ and let $x(\mu_1)$, $x(\mu_2)$ be its partition. If $x(\mu_1) > x(\pi_1')$, the operator $I_{\pi_1', \Lambda^1}$ will annihilate the vector $x(\mu_1)v_{\Lambda^1}$. If, on the other hand, $x(\mu_1) \leq x(\pi_1')$ then $x(\mu_2) > x(\pi_2')$ (since our ordering is compatible with the multiplication). The question is now whether $I_{\pi_2', \Lambda^2}x(\mu_2)v_{\Lambda^2} = 0$. This is actually the point where we need the restriction from $D_{\ell}^{(1)}$ to $D_4^{(1)}$. Namely, the way in which we defined our intertwining operator $I_{\pi_2', \Lambda^2}$ leaves the last successive block with negative colors of $x(\pi_2')$ intact. For instance, if  \\ $x(\pi_2') = x_{\gamma_{\underline{3}}}(-(m-1))x_{\gamma_{\underline{4}}}(-(m-2)) x_{\gamma_2}(-(m-3)) \dots$ we will have:
\begin{align*}
 I_{\pi_2', \Lambda^2} x(\pi_2') v_{\Lambda^2} &= I_{\pi_2', \Lambda^2} x_{\gamma_{\underline{3}}}(-(m-1))x_{\gamma_{\underline{4}}}(-(m-2)) x_{\gamma_2}(-(m-3)) \dots v_{\Lambda^2} = \\ 
&= x_{\gamma_{\underline{3}}}(-(m-1))x_{\gamma_{\underline{4}}}(-(m-2)) I_{\pi_2', \Lambda^2} x_{\gamma_2}(-(m-3)) \dots v_{\Lambda^2} = \\ &= x_{\gamma_{\underline{3}}}(-(m-1))x_{\gamma_{\underline{4}}}(-(m-2)) \cdot C \cdot e^{(m-3) \omega} w_{1234} = \\ &= C'  \cdot e^{(m-1) \omega} w_{12}. 
\end{align*}
Since in monomial $x(\pi)$ we have the exceptional $DC$ after $x(\pi')$ we can't now apply the next basic intertwining operator on vector $I_{\pi_2', \Lambda^2}x(\pi_2')v_{\Lambda^2}$. If $x(\mu_2)$ is now such that it coincides with $x(\pi_2)$ up to the degree $-(m-3)$ and then looks, for example, like this:
$$x(\mu_2) = x_{\gamma_{4}}(-(m-1))x_{\gamma_{\underline{4}}}(-(m-2)) x_{\gamma_2}(-(m-3)) \dots $$
then $I_{\pi_2', \Lambda^2}x(\mu_2)v_{\Lambda^2} \neq 0$ although $x(\mu_2) > x(\pi_2')$. The restriction to $D_4^{(1)}$ enables us now to cover such cases one by one. Namely, since $D_4$ is a ``small" algebra, it is easy to show that if $x(\mu_2) > x(\pi_2)$ and $I_{\pi', \Lambda}x(\mu)(v_{\Lambda^1} \otimes v_{\Lambda^2}) \neq 0$, then necessarily:
\begin{itemize}
 \item [i)] $x(\mu_1) = x(\pi_1')$ and
 \item [ii)] $x(\mu_2)$ is of the same shape like $x(\pi_2')$ and coincides in colors with it up to the factor $-(m-1)$. 
\end{itemize}
Due to the conditions i) and ii) we are left with only two option for $x(\mu_2)$:
\begin{itemize}
\item[1)] $x(\pi_2') = x_{\gamma_{\underline{3}}}(-(m-1))  \dots \quad \textrm{and} \quad x(\mu_2) = x_{\gamma_{\underline{4}}}(-(m-1)) \dots \ $. 
\\ Now we have $$I'_{\pi', \Lambda} (x(\pi_1')v_{\Lambda^1} \otimes x(\pi_2') v_{\Lambda^2})) = C \cdot e^{(m-1) \omega} (w_{123} \otimes w_{124}).$$ Such operator $I'_{\pi', \Lambda}$ will map $x(\mu_1)v_{\Lambda^1} \otimes x(\mu_2)v_{\Lambda_2}$ to $C \cdot e^{(m-1) \omega} (w_{123} \otimes w_{123})$. But the quotient map will send this vector to zero.
\item[2)] $x(\pi_2') = x_{\gamma_{\underline{3}}}(-(m-1))  \dots \quad \textrm{and} \quad x(\mu_2) = x_{\gamma_4}(-(m-1)) \dots  $. \\ 
Here we have:
$$I'_{\pi', \Lambda} (x(\pi_1')v_{\Lambda^1} \otimes x(\pi_2') v_{\Lambda^2})) = C \cdot e^{m \omega} (w_{1234} \otimes w_{12}).$$ In this case, 
$$I'_{\pi', \Lambda} (x(\mu_1)v_{\Lambda^1} \otimes x(\mu_2)v_{\Lambda_2}) = C \cdot e^{m \omega}(w_{1234} \otimes w_{1234})$$ which is again mapped to zero by quotient map.
\end{itemize}
This means that for the operator $I_{\pi',\Lambda} = q \circ I'_{\pi', \\Lambda}$ the condition ii) is also fulfilled.
\end{proof}

Now we need the same proposition for remaining level 2 module, i.e. $L(\Lambda_2)$. 
\begin{proposition}\label{operatori22}
Let $v_{\Lambda_2}$ be a highest weight vector of level 2 module $L(\Lambda_2)$ for algebra $D_4^{(1)}$. Let $x(\pi)$ be a monomial that satisfies $DC$ and $IC$ for $W(\Lambda_2)$. Then there exist a factorization of monomial $x(\pi)$, $x(\pi) = x(\pi_2) x(\pi_1)$ and an intertwining operator $I_{\pi_1, \Lambda_2}$ which commutes with the action of $\gtl_1$ so that the following holds:
\begin{itemize}
 \item[a)] $I_{\pi_1, \Lambda_2}x(\pi_1)v_{\Lambda_2} = C \cdot e^{n \omega} v_{\Lambda'} \neq 0$ for some $n \in \N$ and some $\Lambda'$,
 \item[b)] $I_{\pi_1, \Lambda_2}x(\pi_1')v_{\Lambda_2} = 0$ for $x(\pi_1') > x(\pi_1)$,
 \item[c)] $I_{\pi_1, \Lambda_2}x(\pi_1)v_{\Lambda_2}= C \cdot e^{n \omega} v_{\Lambda'}$ and $x(\pi_2)^{+n}$ satisfies $DC$ and $IC$ on $W(\Lambda')$.
\end{itemize}
\end{proposition}
\begin{proof}
 Let $x(\pi)$ be a monomial that satisfies $DC$ and $IC$ for $W(\Lambda_2)$. If this monomial is of the 
form $x(\pi) = \dots x_{\gamma_{\underline{3}}}(-1)x_{\gamma_3}(-1)$ then the required operator is simply the identity operator since greater monomials have either more $(-1)$ factors (which means they annihilate $v_{\Lambda_2}$) or have $x_{\gamma_2}(-1)$ factor which again annihilates $v_{\Lambda_2}$ (see Subsection \ref{poc_uvj_2}). On the other hand, $x_{\gamma_{\underline{3}}}(-1)x_{\gamma_3}(-1)v_{\Lambda_2} = C \cdot e^{\omega} v_{\Lambda_2}$ (see Remark \ref{izvlacenje}), so the condition c) is also satisfied.

If $x(\pi)$ is not of the form discussed above, this means there exist a first part of it, let's denote it $x(\pi_1)$, such that $x(\pi_1)$ has a partition in two subarrays, $x(\pi_1^1)$ and $x(\pi_1^2)$ that satisfy level 1 conditions. We find this subarrays again by taking every second factor. 

Next, we observe the realization of $v_{\Lambda_2}$ in the tensor product $L(\Lambda_3) \otimes L(\Lambda_3)$, resp. $L(\Lambda_4) \otimes L(\Lambda_4)$. As we know, $v_{\Lambda_2}$ is a linear combination of the basis vectors (see Proposition  \ref{vektor_j}), i.e 
\begin{align}\label{prikaz}
v_{\Lambda_2} = \sum_s C_s \cdot w_{12 \Psi_1^s} \otimes w_{12 \Psi_2^s}
\end{align}
where $\Psi_1^s$ and $\Psi_2^s$ are disjoint sets such that $\Psi_1^s \cup \Psi_2^s = \{ 3,4\}$.

Vector $x(\pi_1^1)w_{12 \Psi_1^s} \otimes x(\pi_1^2) w_{12 \Psi_2^s}$ will be zero for some $s$ but not not for all of them. Now we find those vectors, i.e. vectors $w_{12 \Psi_1^s} \otimes w_{12 \Psi_2^s}$ from \eqref{prikaz} such that $x(\pi_1^1)w_{12 \Psi_1^s} \otimes x(\pi_1^2) w_{12 \Psi_2^s} \neq 0$. In the worst case (see Remark \ref{najgori}), if $x(\pi)$ has the $(-1)$ elements, we are in the one of the following situations:
\begin{itemize}
 \item [a)] $x(\pi_1^1) = \dots x_{\gamma}(-1)$, $x(\pi_1^2) = \dots x_{\delta}(-p)$ where $p \geq 2$ and $\gamma \leq \gamma_3$. We first choose $w_{12 \Psi_1^s}$ such that $x(\pi_1^1)w_{12 \Psi_1^s} \neq 0$. This is done by selectig $\Psi_1^s$ which doesn't contain $r$ if $\gamma_r$ appears in the starting successive block of $x(\pi_1^1)$. After we have chosen $w_{12 \Psi_1^s}$, denote it with $w_{12 \Psi_1^0}$, we take the corresponding $w_{12 \Psi_2^0}$. Then surely $x(\pi_1^2) w_{12 \Psi_2^0} \neq 0$ since $x(\pi_1^2)$ has no elements of $(-1)$ degree, so vector $w_{12 \Psi_1^0} \otimes w_{12 \Psi_2^0}$ has the required property.
 \item[b)] $x(\pi_1^1) = \dots x_{\gamma}(-1)$, $x(\pi_1^2) = \dots x_{\gamma_{\underline{2}}}(-1)$. We select $w_{12 \Psi_1^0} \otimes w_{12 \Psi_2^0}$ as in a). Then again $x(\pi_1^2) w_{12 \Psi_2^0} =\dots x_{\gamma_{\underline{2}}}(-1) w_{12 \Psi_2^0}\neq 0$ (see Lemma \ref{lem_op_1}).
\end{itemize}
Now, for the monomials $x(\pi_1^1)$, $x(\pi_1^2)$ and basic vector $w_{12 \Psi_1^0} \otimes w_{12 \Psi_2^0}$ we define operators $I_{\pi_1^1, \{12 \Psi_1^0 \}}$ and $I_{\pi_1^2, \{12 \Psi_2^0 \}}$ as in Section \ref{operatori}. We need to show that operators selected in this way annihilate other partitions of $x(\pi_1)$ on this and other basic vectors from \eqref{prikaz}. Further on, if $x(\mu_1) > x(\pi_1)$, then those operators should annihilate $x(\mu_1^1)w_{12 \Psi_1^s} \otimes x(\mu_1^2) w_{12 \Psi_2^s}$ for all $s$ from \eqref{prikaz} and all partitions $x(\mu_1^1)$, $x(\mu_1^2)$, of $x(\mu_1)$ as well. So, let $x(\mu_1) \geq x(\pi_1)$ and let $x(\mu_1^1)$ and $x(\mu_1^2)$ be some partition of monomial $x(\mu_1)$. We examine the cases a) and b) separately.

Case a): since $x(\pi_1^2) = \dots x_{\delta}(-p)$ where $p>2$, on the second component we apply the basic intertwinig operator which sends $w_{12 \Psi_2^0}$ to $e^{\omega}v_{\Lambda_0}$. This operator will annihilate any other $w_{12 \Psi_2^s}$ different from  $w_{12 \Psi_2^0}$. So, the only basic vector from \eqref{prikaz} that can ''survive" is $w_{12 \Psi_1^0} \otimes w_{12 \Psi_2^0}$. 

Case b): here we have $x(\pi_1^2) = \dots x_{\gamma_{\underline{2}}}(-1)$, i.e. \\ $x(\pi_1)w_{12 \Psi_1^0} \otimes w_{12 \Psi_2^0} = \dots (x_{\gamma}(-1) w_{12 \Psi_1^0} \otimes x_{\gamma_{\underline{2}}}(-1)w_{12\Psi_2^0}) $. We apply on the second factor therefore the basic intertwining operator which sends $x_{\gamma_{\underline{2}}}(-1)w_{12\Psi_2^0} = C \cdot e^{\omega}w_{1 \Psi_2^0}$ to $C' \cdot e^{2\omega}v_{\Lambda_0}$.  

Since $x(\mu_1) \geq x(\pi_1)$, monomial $x(\mu_1)$ has to have two $(-1)$ factors as $x(\pi_1)$ does (if it had more it would annihilate $v_{\Lambda_2}$). So, we have:
\begin{align}\label{baz_vekt}
& x(\mu_1) v_{\Lambda_2} = \sum x(\mu_1^1) w_{12 \Psi_1^s} \otimes x(\mu_1^2)w_{12 \Psi_2^s} = \\ \notag & = \sum \dots (x_{\delta_1}(-1) w_{12 \Psi_1^s} \otimes x_{\delta_2}(-1)w_{12 \Psi_2^s})
\end{align}
where sum goes over $s$ and all the partitions of $x(\mu_1)$.

Now, if $x_{\delta}(-1)w_{12 \Psi_2^s} = D \cdot e^{\omega} w_{1 \Psi_2^0}$ for some $\delta \in \Gamma_1$ and some $s$, then necessarily $\delta = \gamma_{\underline{2}}$ and $w_{12 \Psi_2^s}=w_{12 \Psi_2^0}$.
So, if we want $I_{\pi_1^2, \ \{12 \Psi_2^0 \} } x_{\delta_2}(-1) w_{12 \Psi^s} \neq 0$, then we have to have $\delta = \gamma_{\underline{2}}$ and $w_{12 \Psi_2^s} = w_{12 \Psi_2^0}$. Therefore we conclude that the only basis vector from \eqref{baz_vekt} that will ``survive" after we apply operator $I_{\pi_1^2, \ \{12 \Psi_2^0 \} }$ is again $w_{12 \Psi_1^0} \otimes w_{12 \Psi_2^0}$.
\vspace{2mm}

Altogether, cases a) and b) show us if some vectors $x(\mu_1^1)w_{12 \Psi_1^s} \otimes x(\mu_1^2) w_{12 \Psi_2^s}$ differ from zero after the action of the operator, then it must be $x(\mu_1^1)w_{12 \Psi_1^s} \otimes x(\mu_1^2) w_{12 \Psi_2^s} =  x(\mu_1^1)w_{12 \Psi_1^0} \otimes x(\mu_1^2) w_{12 \Psi_2^0}$. This means we actually analyze $x(\pi_1)$ and $x(\mu_1)$ solely on $w_{12 \Psi_1^0} \otimes w_{12 \Psi_2^0}$ and the existence of the required operators follows now from the previous proposition.
\end{proof}

\section{Proof of linear independence in level 2 case}\label{kraj2}
Now we can state the main result of this paper.

\begin{theorem}
Let $L(\Lambda)$ be a standard module of level two for algebra $D_4^{(1)}$ and let $v_{\Lambda}$ be its highest weight vector. The set of monomial vectors
\begin{align} \label{baza_2}
\{ \ x(\pi)v_{\lambda} \ \vert \ x(\pi) \ \textrm{satisfies $DC$ and $IC$ on $W(\Lambda)$} \}
\end{align}
is a basis of $W(\Lambda)$.
\end{theorem}
\begin{proof}
We know (see Propositions \ref{PU_UR_2} and \ref{PU_UR_22}) that set of admissible vectors \eqref{baza_2} spans $W(\Lambda)$. It is left to show the linear independence of this vectors. The proof is the same as for level one, we carry it out for all modules of level two simultaneously, by induction on the degree of the monomials. 

If a monomial $x(\pi)$ satisfies $DC$ and $IC$ on $W(\Lambda)$ then by Propositions \ref{operatori21} and \ref{operatori22} there exist a partition of the monomial $x(\pi)$, $x(\pi) = x(\pi_2) x(\pi_1)$ and an intertwining operator $I_{\pi_1, \Lambda}$ which commutes with the action of $\gtl_1$ such that the following holds:
\begin{itemize}
 \item[a)] operator $I_{\pi_1, \Lambda}$ acts on  $x(\pi_1)$ and $I_{\pi}x(\pi_1)v_{\Lambda_2} = C \cdot e^{n \omega} v_{\Lambda'} \neq 0$ for some $n \in \N$ and some $\Lambda'$,
 \item[b)] $I_{\pi_1, \Lambda}x(\pi_1')v_{\Lambda_2} = 0$ for $x(\pi_1') > x(\pi_1)$,
 \item[c)] $I_{\pi_1, \Lambda}x(\pi_1)v_{\Lambda_2}= C \cdot e^{n \omega} v_{\Lambda'}$ and $x(\pi_2)^{+n}$ satisfies $DC$ and $IC$ on $W(\Lambda')$.
\end{itemize}
Suppose now the theorem is true for the monomials of degree greater than $-n$. Let
\begin{align}\label{sumica}
\sum_{\mu} c_{\mu} x(\mu) v_{\Lambda} =0
\end{align}
where all the monomials are of degree $\geq -n$. Fix $x(\pi)$ in \eqref{sumica} and suppose $c(\mu) =0$ for $x(\mu) < x(\pi)$. We apply operator $I_{\pi_1, \Lambda}$ to the \eqref{sumica} and get:
\begin{align}
 I_{\pi_1, \Lambda} \left( \sum_{x(\mu_1) > x(\pi_1)}  c_{\mu}x(\mu) \right) v_{\Lambda} &+ I_{\pi_1, \Lambda} \left( \sum_{x(\mu_1) =  x(\pi_1)} c_{\mu}x(\mu) \right) v_{\Lambda} + \\ &+ I_{\pi_1, \Lambda} \left( \sum_{x(\mu_1) < x(\pi_1)}  c_{\mu}x(\mu) \right) v_{\Lambda} =0
\end{align}
In the third sum, the coefficients are by assumption equal to zero. In the first, by b) we have $I_{\pi_1, \Lambda}x(\mu_1)v_{\Lambda}=0$. What is left is:
$$I_{\pi_1, \Lambda} \left( \sum_{x(\mu_1) =  x(\pi_1)} c_{\mu}x(\mu) \right) v_{\Lambda} =e(n \omega) \cdot C \cdot \left( \sum_{x(\mu_1) =  x(\pi_1)} c_{\mu}x(\mu_2)^{+n} \right) v_{\Lambda'} =0$$
Since $e(n \omega)$ is an injection we conclude:
$$\left( \sum_{x(\mu_1) =  x(\pi_1)} c_{\mu}x(\mu_2)^{+n} \right) v_{\Lambda'} =0$$
The monomials $x(\mu_2)^{+n}$ do not necessary satisfy $IC$ on $v_{\Lambda'}$. If this is the case, we proceed on the next block, i.e. we look now at the $x(\pi_2)^{+n}v_{\Lambda'}$ and we find the corresponding partition, $x(\pi_2)^{+n} = x(\pi_2^2)x(\pi_2^1)$ and the operator, denote it $I_{\pi_2, \Lambda'}$, for the monomial $x(\pi_2)^{+n}$ such that a), b) and c) hold. We get:
\begin{align}
 I_{\pi_2, \Lambda'} \left( \sum_{\substack{x(\mu_1) =  x(\pi_1) \\  x(\mu_2^{1}) > x(\pi_2^1)}}  c_{\mu}x(\mu_2)^{+n} \right) v_{\Lambda'} &+ I_{\pi_2, \Lambda'} \left( \sum_{\substack{x(\pi_1') =  x(\pi_1) \\  x(\mu_2^{1}) = x(\pi_2^1)}} c_{\mu}x(\mu_2)^{+n} \right) v_{\Lambda'} + \\ &+ I_{\pi_2, \Lambda'} \left( \sum_{\substack{x(\mu_1) =  x(\pi_1) \\  x(\mu_1^2) < x(\pi_1^2)}}  c_{\mu}x(\mu_2)^{+n} \right) v_{\Lambda'} = 0
\end{align}
Here again only the second sum remains. Let $I_{\pi_2, \Lambda'}x(\pi_2^1)v_{\Lambda'} = C \cdot e^{m\omega} v_{\Lambda''}$. We commute $e^{m \omega} = D \cdot e(m \omega)$  and get 
$$\left( \sum_{\substack{x(\mu_1) =  x(\pi_1) \\  x(\mu_1^2) = x(\pi_1^2)}}  c_{\mu}x(\mu_2^2)^{+n+m} \right) v_{\Lambda''} =0.$$
If all remaining monomials satisfy $IC$ for $W(\Lambda'')$ we apply the assumption of the induction. If not, we continue in the same way.
\end{proof}

\end{document}